\documentclass{aims} 
\usepackage{amsmath,comment}
\usepackage{paralist}
\usepackage[misc]{ifsym}
\usepackage{epsfig} 
\usepackage{epstopdf} 
\usepackage[colorlinks=true]{hyperref}
\hypersetup{urlcolor=blue, citecolor=red}
\allowdisplaybreaks

\usepackage{tikz}

\def\currentvolume{X}
 \def\currentissue{X}
  \def\currentyear{200X}
   \def\currentmonth{XX}
    \def\ppages{X-XX}
     \def\DOI{10.3934/xx.xxxxxxx}

 \usepackage{xcolor}

\renewcommand{\epsilon}{\varepsilon}

\newcommand{\R}{\mathbb R}

\newcommand\dps{\displaystyle}
\newcommand\be{\begin{equation}}
\newcommand\ee{\end{equation}}
\newcommand\bea{\begin{eqnarray}}
\newcommand\eea{\end{eqnarray}}
\newcommand\beaa{\begin{eqnarray*}}
\newcommand\eeaa{\end{eqnarray*}}
\newcommand\bay{\begin{array}}
\newcommand\eay{\end{array}}
\newcommand\ba{\begin{align}}
\newcommand\ea{\end{align}}

\newcommand{\red}[1]{{#1}}
\newcommand{\blue}[1]{{#1}}

\newtheorem{theorem}{Theorem}[section]
\newtheorem{corollary}[theorem]{Corollary}

\newtheorem{lemma}[theorem]{Lemma}

\newtheorem{problem}{Open Question}
\newtheorem*{question}{Question}

\theoremstyle{definition}
\newtheorem{definition}[theorem]{Definition}
\newtheorem{remark}[theorem]{Remark}

\newcommand{\ep}{\varepsilon}

\title[Invasion Fronts via Hamilton-Jacobi Approach]
{Invasion Fronts in Shifting Habitats and Competition Systems: A Hamilton-Jacobi Approach and Nonlocal Effects} 

\author[K.-Y. Lam,  C.-H. Wu and X. Yu]{}

\thanks{The first author is partially supported by National Science Foundation (Grant DMS-2325195,\,DMS-2602291).
CHW is supported by the NSTC of Taiwan (NSTC 113-2628-M-A49-004-MY4) and the National Center for Theoretical Sciences.
XY is supported in part by the NSFC (Nos. 12371169 and 12571177).}

\thanks{$^*$Corresponding author: King-Yeung Lam (\href{mailto:lam.184@osu.edu}{lam.184@osu.edu})
}
\thanks{
The authors are grateful to Professor Fran\c{c}ois Hamel and Professor Lionel Roques for inviting them to contribute this survey article to the thematic session ``Reaction--Diffusion Equations: New Theoretical and Applied Trends'' at the 15th AIMS Conference, held in Athens, Greece.
}

\begin{document}
\maketitle
\centerline{\scshape
King-Yeung Lam$^{{\href{mailto:lam.184@osu,edu}{\textrm{\Letter}}}*1}$, 
Chang-Hong Wu$^{{\href{mailto:changhong@math.nctu.edu.tw}{\textrm{\Letter}}}2}$
and Xiao Yu$^{{\href{mailto:xymath19@m.scnu.edu.cn}{\textrm{\Letter}}}3}$}

\medskip

{\footnotesize
 \centerline{$^1$Department of Mathematics, The Ohio State University}
} 

\medskip

{\footnotesize
 \centerline{$^2$Department of Applied Mathematics, National Yang Ming Chiao Tung University}
}

\medskip

{\footnotesize
 \centerline{$^3$School of Mathematical Sciences, South China Normal University}
}

\bigskip

 \centerline{(Communicated by Handling Editor)}




\begin{abstract}
This survey reviews recent developments in the study of spreading phenomena in
reaction--diffusion equations arising from ecological invasion models.
Motivated by the conjecture of Shigesada and Kawasaki on staged invasions,
we discuss how competition systems can lead to effective scalar models
with shifting habitats. We present the Hamilton--Jacobi approach for
determining spreading speeds and revisit the result of 
\red{Li--Bewick--Shang--Fagan} (2014)
from this perspective. We then describe the emergence of nonlocally pulled
fronts when the shifting habitat connects regions of distinct positive growth rates.
Recent results including the works of Lam--Yu (2022) and Lam--Nadin--Yu (2025)
are surveyed. We also discuss spreading phenomena in competition systems, \red{predator-prey systems}, and the existence of various classes of entire solutions. Finally, we discuss the logarithmic 
correction for invasion waves in moving environments and prove a new result.
\end{abstract}

\tableofcontents

\section{Ecological Motivation and the Shigesada--Kawasaki Conjecture}  
The spatial spread of biological species into new environments is a central problem in ecology. A well-known example is the post-glacial northward migration of tree species across North America \cite{davis1981quaternary}, where different species expanded their ranges over time into previously uninhabitable regions.

Such phenomena can be modeled mathematically as the propagation of invasion fronts into unstable states \cite{avery2025pulledfrontsjustpulled,van2003front}. In this context, an unstable state typically represents the absence of a species, while the propagating front describes the expansion of a population into an initially empty but favorable habitat.

In case of a single species, a classical model for this process is the nonlinear diffusion equation
\begin{equation}\label{e.pde1}
u_t = u_{xx} + f(u) \quad \text{ for }t>0,~x \in \mathbb{R},
\end{equation}
introduced in the seminal works of Fisher \cite{fisher1937wave} and of Kolmogorov, Petrovsky, and Piscounoff \cite{KolmogorovPetrovskyPiskunov1937}. Such a class of models, originally motivated by population genetics and ecology, laid the foundation for the mathematical theory of front propagation. 

A central question in this theory is to determine how fast an invading population spreads. This leads to the notion of spreading speed formalized in \cite{aronson1978multidimensional,weinberger1982long}, which characterizes the asymptotic rate at which the invasion front advances.



\begin{definition}[Spreading speed] \label{spread}
Let $u: (0,\infty)\times\mathbb R \to [0,\infty)$  be the density function of a given population, typically arising from front-like or compactly supported initial data.  

A number $c^*>0$ is called the (rightward) \emph{spreading speed} associated with a class of initial data,  if for each solution $u(t,x)$ with such initial data,  the following two properties hold:
\[
\begin{cases}
\lim\limits_{t\to\infty}\sup\limits_{x\ge ct} u(t,x)=0 \qquad &\text{ for every }c > c^*;\\
\liminf\limits\limits_{t\to\infty}\inf\limits_{0\le x\le ct} u(t,x)>0 &\text{ for every }
0<c<c^*.
\end{cases}
\]
\end{definition}



\blue{
A separate approach to model range expansion involves models with evolving boundaries, e.g. of stefan-type.
    Since we will focus mainly on the propagation of level sets in Cauchy problems, we refer the readers to the survey \cite{Du2022Propagation} for a recent account, and to \cite{MaDuWang2026} for the connection between  Cauchy problems and a new class of problems involving a ``preferred population density" on the moving boundary. 
}

\medskip

\noindent {\bf From Single-Species to Multi-Species Invasion.}  
While the spreading behavior of a single species for \eqref{e.pde1} is now well understood, considerably less is known about the interaction of multiple invading species.

Motivated by ecological observations, Shigesada and Kawasaki proposed the study of co-invasion in open habitats \cite{shigesada1997biological}. Here, an open habitat refers to an environment that allows reproductive growth once invaded. 
To model this phenomenon, they considered the classical Lotka–Volterra competition–diffusion system
\begin{equation}\label{e.lv}
\begin{cases}
u_t = u_{xx} + u(1-u-av), \\
v_t = d v_{xx} + r v(1-v-bu),
\end{cases}
\end{equation}
where $u$ and $v$ represent the population densities of two competing species, and $d,r,a,b$ are positive constants. Here, one is interested in studying the co-invasion phenomenon of several competing species where the initial data of every component is compactly supported (or more generally, front-like). This differs from classical studies (see, e.g., \cite{lewis2002spreading}) on \eqref{e.lv}, which typically consider initial data that are
compact perturbations of the semi-trivial equilibria $E_u=(1,0)$ or $E_v = (0,1)$.

Their numerical investigations suggested a striking phenomenon known as staged invasion, in which: 
\begin{itemize}
\item a fast disperser invades the empty territory,
\item a slower competitor subsequently follows,
\item the second invasion speed may depend on the first.
\end{itemize}

Understanding how these invasion speeds interact remains a fundamental open problem and serves as the main motivation for this survey.


\medskip

\noindent {\bf Reduction to a Single-Species Model with Shifting Habitat.} 
To make the key issue more transparent, we first discuss a reduced problem. When the faster species spreads into the empty territory with speed $c_1>0$, the environment
experienced by the slower species can be approximated by a spatially shifting
growth rate. Indeed, if we formally replace $v = \mathbf{1}_{\{x\le c_1 t\}}$ 
in the first equation of 
\eqref{e.lv}, then one obtains a PDE incorporating a shifting variable $x-c_1t$:
\begin{equation}\label{e.0312.1}
u_t = u_{xx} + f(x-c_1 t,u),
\end{equation}
which  
reduces to the case when the growth rate changes across the moving interface $x=c_1 t$:
\begin{equation}\label{e.0325.2}
f(x-c_1t,u)/u = 1 - a\,\mathbf{1}_{\{x-c_1 t\le 0\}} - u.
\end{equation}
This is a special case belonging to a class of models with {\it shifting habitat}.

Next, we discuss recent progress to the following conceptual question:
\begin{question}
    How do the environmental shifting speed $c_1$ and the initial data jointly determine the persistence or extinction and the spreading speed of the population?
\end{question}


\section{Reaction--Diffusion Equations in Shifting Habitats}\label{s.2}

The type of spatio-temporal heterogeneity incorporating the variable $x-c_1t$ in \eqref{e.0312.1} was introduced in a separate context in the work of
Potapov and Lewis \cite{potapov2004climate} and Berestycki et al. \cite{berestycki2009can}, where they studied the effect of the shifting isotherms on biota.

Since then, models of propagation in shifting habitats has been developed to treat various modeling contexts, including discrete diffusion, nonlocal diffusion, and others; 
See \cite{wang2022recent} for a recent account of the literature. Here we aim to present some tools and apply them to the simplest reaction-diffusion models, to make the ideas transparent. However, we expect these tools to be useful in other modeling contexts as well. 

By assuming that the size of a bounded moving source patch
for a focal species is fixed and is  surrounded by sink patches, 
the critical patch size for species persistence was analyzed  
\cite{berestycki2009can,lam2024population,potapov2004climate}. In \cite{li2014persistence}, Li et al. proposed to study \eqref{e.0312.1} with a Fisher-KPP nonlinearity:
\begin{equation}\label{e.0312.2}
f(x-c_1 t,u) = (r(x- c_1 t) - u)u
\end{equation}
and studied the situation in which the favorable environment is unbounded and retreating in the sense
that $c_1>0$ and 
\begin{equation}\label{e.0312.3}
r \in C(\mathbb{R})\text{ is increasing and satisfies }r(-\infty) < 0 < r(+\infty).
\end{equation}
\begin{theorem}\label{t.0312.1}
{\rm (\cite[Theorems 2.1 and 2.2]{li2014persistence})}
Let $u$ be a solution of \eqref{e.0312.1} with bounded, nonnegative, and nontrivial compactly supported initial data. 
Suppose \eqref{e.0312.2}--\eqref{e.0312.3} hold. 
\begin{itemize}
    \item[{\rm(a)}] If $c_1 > 2\sqrt{r(+\infty)}$, then 
    \[
    \sup_{x\in \mathbb{R}} u  \to 0 \quad \text{ as } t\to+\infty.
    \]
    \item[{\rm(b)}] If $0<c_1 < 2\sqrt{r(+\infty)}$, then for each  $\eta\in\left(0, \tfrac{2\sqrt{r(+\infty)} - c_1}{2}\right)$,
\[
\begin{aligned}
\lim_{t\to\infty}\red{\sup\limits_{x\in\R}}\Big[&|u|\,\mathbf{1}_{\{x\le t(c_1-\eta)\}}
      + |u-r(+\infty)|\,\mathbf{1}_{\{t(c_1+\eta)\le x\le t(2\sqrt{r(+\infty)}-\eta)\}} \\
   & \qquad + |u|\,\mathbf{1}_{\{x\ge t(2\sqrt{r(+\infty)}+\eta)\}}
\Big] = 0 .
\end{aligned}
\]
\end{itemize}
\end{theorem}
\begin{proof}[First proof of Theorem \ref{t.0312.1} (via super/sub-solutions method).]
Here, we sketch an elementary proof using the method of super- and subsolutions. 

First, let $r(+\infty)>0$ and $c_1>0$ be given constants, and 
set $\lambda_1=\sqrt{r(+\infty)}$, then
 the supersolution $\overline{u}_1(t,x):= C_1e^{-\lambda_1(x - 2\lambda_1 t)}$ (for some $C_1> 1$)  
implies that 
\begin{equation}\label{e.0312.4a}
     \lim_{t\to\infty} \sup_{x \geq t(2\sqrt{r(+\infty)}+\eta)} |u| =0 \qquad \text{for every }\eta >0.
 \end{equation}
Next, we claim that 
 \begin{equation}\label{e.0312.4b}
     \lim_{t\to\infty} \sup_{x \leq t(c_1-\eta)} |u| =0 \qquad \text{for every }\eta >0.
 \end{equation}
 Suppose not. Then, thanks to parabolic regularity, we can pass to a sequence $(t_n,x_n)$ such that $t_n \to \infty$, $v_n(t,x):= u(t + t_n, x + x_n)$ converges in $C_{loc}(\mathbb{R}^2)$ to some $\bar v(t,x)$ such that  $\bar{v}(0,0)>0$ and satisfies
 \begin{equation}\label{e.0312.4bb}
 \begin{cases}
\bar v_t = \bar v_{xx} + r(-\infty) \bar v - \bar v^2 \quad &\text{ for }(t,x) \in \mathbb{R}^2,\\
0 \leq \bar v \leq {r(+\infty)} &\text{ for }(t,x) \in \mathbb{R}^2.
\end{cases}
 \end{equation}
However, \eqref{e.0312.4bb} and $r(-\infty) <0$ imply that $\bar v \equiv 0$, which is incompatible with $\bar{v}(0,0) >0$. This contradiction establishes \eqref{e.0312.4b}.

Now, assertion (a) is a consequence of \eqref{e.0312.4a} and \eqref{e.0312.4b}. To prove assertion (b), \red{for each $\eta>0$, fix any $c \in [c_1+\eta, 2\sqrt{r(+\infty)}-\eta]$}, and consider (see \cite{lam2022introduction})
\[
\underline{u}_1(t,x)= \begin{cases}
    \epsilon e^{- \frac{c}{2}(x-ct)}\sin \left(\frac{x-ct}{R} \right) &\text{ when }0 < \frac{x-ct}{R} <\pi,\\
    0 &\text{ otherwise}.
\end{cases}
\]
For any $R>0$ such that $r(+\infty) > \frac{c^2}{4} + \frac{3}{R^2}$, there exists $T_R>0$ such that for any $\epsilon \in (0, \frac{1}{R^2})$, one can show \cite[Definition 1.1.1 and Lemma 6.3.7]{lam2022introduction} that $\underline{u}_1$ is a  generalized subsolution of  \eqref{e.0312.1} in the domain $(T_R, \infty) \times \mathbb{R}$. By comparing $u$ with $\underline{u}_1$ (with \red{suitably chosen $\epsilon$, $R$, $T_R$ independent of $c$}), we deduce that 
\begin{equation}\label{e.persist0723}
\liminf_{t\to\infty} \inf_{t(c_1 + \eta) \leq x \leq t(2\sqrt{r(+\infty)} - \eta)} u >0 \, \text{ for each }0 < \eta < \tfrac{2\sqrt{r(+\infty)} - c_1}{2}.
\end{equation}
Arguing as before using classification of entire solutions (which are bounded from below by a positive constant) to the constant coefficient problem \eqref{e.0312.4bb} with $r(+\infty)$ replacing $r(-\infty)$, one can show
\begin{equation}\label{e.0312.4c}
   \lim_{t\to\infty} \sup_{t(c_1 + \eta) \leq x \leq t(2\sqrt{r(+\infty)} - \eta)} |u - r(+\infty) |=0  \, \text{ for each }0 < \eta < \tfrac{2\sqrt{r(+\infty)} - c_1}{2}. 
\end{equation}
This completes the proof.
\end{proof}


\medskip

\noindent {\bf Forced traveling waves in shifting environments.} 
Since there is an intrinsic speed $c_1$ in the environment, one question is the existence of the so-called {\it forced waves} which are the solutions that maintain the same profile but shift at a speed exactly equal to the speed $c_1$.  

\begin{definition}[Forced Traveling Waves]
    We say that $U: \mathbb{R} \to \mathbb{R}$ is a {\it forced traveling wave} of \eqref{e.0312.1} if $\lim\limits_{\xi \to \pm \infty} U(\xi)$ exist, and 
\[
-c_1 U'(\xi) =     U''(\xi) +  f(\xi, U(\xi)) \quad \text{ for } \, \xi \in \mathbb R.
\]
\end{definition}

\blue{This question was first studied by Hamel \cite{hamel1997reaction1,hamel1997reaction} in cylinder domains when the environment is non-increasing for Fisher-KPP, ignition and bistable type problems. In one
dimensional space, the existence, multiplicity of forced waves when $z\mapsto f(\pm\infty, z)$ is of Fisher-KPP type, as well as their attractivity for the initial value problem are considered in \cite{Fang2016can,berestycki2018forced,Tang2026stability}.} 
For some related recent work, see  \cite{yuan2019spatial} for two competing species and \red{\cite{choi2022forced,Giletti2023}} for predator-prey systems.


For the  Fisher--KPP equation with constant coefficients, it is well known that there exists a family of  traveling wave solutions 
parameterized by the wave speed \(c\ge c^*\).
When a shifting habitat is introduced, we propose to study, more generally, the class of entire solutions (those solutions that are defined for $(t,x) \in \mathbb{R}^2$) to understand how the species invasion interacts with the shifting environment.

\begin{problem}
Consider \eqref{e.0312.1} such that 
$f(y,s) \leq f_u(y,0)s$ for $s \geq 0$ and $y\in\mathbb{R}$. Then there exists $\tilde\lambda>0$ such that for each $\lambda \in (0,\tilde\lambda)$, \eqref{e.0312.1} has a unique positive entire solution $U_\lambda \in C^{1,2}(\mathbb{R}^2)$ such that for each $t$
\[
\lim_{x \to -\infty} U_\lambda(t,x) = 0, \quad U_\lambda(t,x) \sim e^{-\lambda x} \text{ as }  {x\to+\infty}.
\]
In case \eqref{e.0312.2} and \eqref{e.0312.3} hold, we can take  $\tilde \lambda= \sqrt{r(+\infty)}$ 
and we have
\[
\lim_{\ep \to 0} U_\lambda\left(\frac{t}{\ep}, \frac{x}{\ep} \right) = \begin{cases}
    0 &\text{ in } \{t<0\}\cup \{t>0,~x<t c_1\} \cup \{t>0,~x > t c_\lambda\},\\
    r(+\infty) &\text{ in } 
    \{t>0,~t c_1<x<tc_\lambda \},
\end{cases}
\]
where $c_\lambda = \lambda + r(+\infty)/\lambda$ if $\lambda \in (0, \sqrt{r(+\infty)})$ and equals to $2\sqrt{r(+\infty)}$ otherwise.
\end{problem}
It would be interesting to characterize the family of entire solutions for monotone or nonmonotone environmental functions $r$, which contains the class of forced traveling waves as a proper subset.



\section{Hamilton--Jacobi Approach to Spreading}

The Hamilton--Jacobi approach was introduced by Freidlin \cite{Freidlin1985limit,freidlin1986geometric}, who
employed probabilistic arguments to study the asymptotic behavior of
solutions to the Fisher--KPP equation modeling the population of a
single species. Subsequently, the result was generalized by Evans and
Souganidis \cite{Evans1989pde} using PDE arguments based on the theory of viscosity solutions of Hamilton-Jacobi type PDEs \cite{crandall1983viscosity,crandall1984some}; See \cite{barles2024modern} for the state-of-the-art of general theory. This method clarifies  how the solutions are driven by the underlying linear problem.


In order to study the asymptotic speed of spread, we take the hyperbolic limit to separate the problem of evaluation of velocity from that of the precise wave form,
\begin{equation}\label{e.wkb}
u^\ep(t,x) = u \left( \frac{t}{\ep},\frac{x}{\ep}\right),
\end{equation}
and we introduce a WKB ansatz to describe the thin front ahead of the invasion wave via the {\it rate function}:
\[
w^\ep(t,x)= -\ep\log u^\ep(t,x).
\]
Information about the spreading behavior of $u^\ep$ is encoded in the asymptotic limit of the rate function $w^\ep$. In the special case of spreading in constant environments, we state the following result which a special case of the theory in \cite{freidlin1986geometric,Evans1989pde}.
%
\begin{theorem}[Spreading in constant environments]
Let $u$ be a solution to \eqref{e.0312.1}  with
nonnegative, compactly supported initial data. Assume the nonlinearity is spatially homogeneous, i.e. $f=f(s)$ in \eqref{e.0312.1} and such that
\[
f'(0)>0\qquad  \text{ and }\qquad \sup_{s >0}f(s){s^{-1}} \leq f'(0). 
\]
Then
as $\ep \to 0$, $w^\ep(t,x)$ converges to some limit $\hat w(t,x)$ in $C_{loc}((0,\infty)\times \mathbb{R})$ and 
\[
\lim_{\ep \to 0} u^\ep 
\begin{cases}
    \red{= } 0 & \text{ uniformly on compact subsets of }\{\hat w>0\},\\
    >0 & \text{ uniformly on compact subsets of } {\rm{Int}}\,\{\hat w=0\}.
\end{cases}
\]
Moreover, $\hat w$ is the unique viscosity solution of 
\begin{equation}\label{e.hje.1}
\min\{w, w_t + H(w_x) \} = 0 \qquad \text{ in }(0,\infty)\times\mathbb{R}
\end{equation}
satisfying the initial condition
\red{\begin{equation}\label{e.ini1}w(0,x) =  \begin{cases}0 &\text{ for } x=0,\\
+\infty &\text{ for } x \neq 0.
\end{cases}
\end{equation}
understood in the relaxed sense; see Remark~\ref{rmk.0725.1},}
where $H(\lambda) = \lambda^2 + f'(0)$.
\end{theorem}
\blue{
\begin{remark}
Condition \eqref{e.ini1} holds in the sense that the half-relaxed limits (see \eqref{e.halfrelax}) satisfies 
\[
\begin{cases}
\overline{w}(0,x)=\limsup\limits_{\ep \to 0\atop (t',x') \to (0,x)} w^\ep(t',x')= +\infty, \quad \\
\underline{w}(0,x)=\liminf\limits_{\ep \to 0\atop (t',x') \to (0,x)} w^\ep(t',x')=\begin{cases}
 0 &\text{ if }x = 0,\\
 +\infty  &\text{ if }x \neq 0.
\end{cases}
\end{cases}
\]
This can be derived directly from \eqref{e.wkb} and we omit the details.
\label{rmk.0725.1}
\end{remark}
\begin{remark}
    By constructing a  stationary weak subsolution of the form $\delta \sin(x/L)$ for some $\delta,L>0$, one can further show that $\overline{w}(t,0) = \underline{w}(t,0) = 0$ for all $t >0$. In this sense, the rightward spreading speed can be alternatively studied by considering the initial data $w(0,x) = 0$ for $x <0$ and $w(0,x) = +\infty$ for $x >0$, or by considering the limiting problem in a sector; See Section  \ref{s.4}.
\end{remark}
}

For the reader's convenience, we recall the definition of a viscosity solution \cite{CrandallIshiiLions1992}.
\begin{definition}[Viscosity solution]
Let $H:\mathbb{R}\to\mathbb{R}$ be continuous and let $\Omega$ be an open set in $\mathbb{R}^2$. Consider the Hamilton-Jacobi equation
\begin{equation}\label{e.0315.hje}
\min\{w,\, w_t + H(w_x)\}=0 
\quad \text{in } \Omega. \end{equation}

\begin{itemize}
\item[{\rm(i)}] Let $w: \Omega \to [-\infty, \infty)$ be an upper semicontinuous function on $\Omega$. We say that $w$ is a viscosity subsolution of \eqref{e.0315.hje} if  
for any $\phi\in C^1(\Omega)$ and any 
$(t_0,x_0)\in(0,\infty)\times\mathbb{R}$ such that 
$w-\phi$ attains a strict local maximum at $(t_0,x_0)$ and $w(t_0,x_0)>0$, we have
\[
\phi_t(t_0,x_0) + H(\phi_x(t_0,x_0)) \le 0.
\]

\item[{\rm(ii)}]
Let $w: \Omega \to (-\infty, \infty]$ be a lower semicontinuous function on $\Omega$. We say that $w$ is a viscosity supersolution of \eqref{e.0315.hje} if   $w \geq 0$ in $\Omega$, and  
for any $\phi\in C^1((0,\infty)\times\mathbb{R})$ and any 
$(t_0,x_0)\in(0,\infty)\times\mathbb{R}$ such that 
$w-\phi$ attains a strict local minimum at $(t_0,x_0)$, we have
\[
\phi_t(t_0,x_0) + H(\phi_x(t_0,x_0)) \ge 0.
\]

\item[{\rm(iii)}] If $w$ is both a viscosity subsolution and a viscosity supersolution, 
then $w$ is called a viscosity solution.
\end{itemize}
\label{d.0409.1}
\end{definition}

\medskip


\noindent {\bf The Lagrangian Formulation and Selection of Speed Based on Local versus Nonlocal Information.} 
When $H$ is strictly convex and coercive,  
the function \(w\) admits the following control (or ``least action principle") formulation \cite{barles2024modern}:
\begin{equation}\label{e.lagrangian}
w(t,x)
=
\max\left\{
0,\;
\inf_{\substack{\gamma \in AC([0,t])\\ \gamma(0)=0,\ \gamma(t)=x}}
\int_{0}^{t} L(\dot{\gamma}(s))\,ds
\right\},
\end{equation}
where the function $L$ is the Legendre transform of $H$, which is given by
\begin{equation}\label{e.legendre}
L(v):= \max_{\lambda \in \mathbb{R}} \left[\lambda v - H(\lambda)\right].
\end{equation}
\begin{remark}
\blue{For general initial data $w(0,x) = w_0(x)$, the control formulation is
\begin{equation}\label{e.0725.1}
w(t,x)
=
\max\left\{
0,\;
\inf_{\substack{\gamma \in AC([0,t])\\  \gamma(t)=x}}
\int_{0}^{t} L(\dot{\gamma}(s))\,ds + w_0(\gamma(0))
\right\}.
\end{equation}
In case of the initial data specified in \eqref{e.ini1}, the class of admissible paths are reduced to the ones satistying $\gamma(0)=0$ and $\gamma(t)=x$.}   
\end{remark}


  From a genealogical point of view, the optimal path can be seen as an ancestral
lineage in an asexually reproducing population, where the Lagrangian $L$ represents the reproductive cost. Consequently, the lineage that minimizes the cost corresponds to
the typical ancestry of individuals located at position $x$ at time $t$; See \cite{florien2022ancestral} for a derivation based on ancestral stochastic processes.

To evaluate $w(t,x)$ for fixed $t>0$ and $x>0$, we consider all possible trajectories $\gamma(s)$ that an individual could have followed to reach $x$ at time $t$. Along each trajectory, the individual experiences both growth (as a payoff) and dispersal (as a cost), which together determine the cumulative cost functional $\int_0^t L(\dot\gamma(s))\,ds$. The value $w(t,x)$ is then the minimal cumulative cost over all such trajectories. 
In biological terms, this means that the population observed at $(t,x)$ is dominated by the lineage whose ancestral path minimizes the cost. 


When the Lagrangian $L(v)$ is a convex function of $v$, one can show that the minimizing path is the straight line
\begin{equation}\label{e.0412.1}
\gamma^*(s)=\frac{x}{t}s,    
\end{equation}
which is the basis of the Lax-Oleinik formula
\[
\red{w(t,x)=\max\left\{0,\int_0^t L(\dot{\gamma}^*(s))\,ds\right\}
=
\max\left\{0, t\,L\!\left(\frac{x}{t}\right)\right\}.}
\]
The above Lagrangian formulation can be interpreted as follows:
in a homogeneous environment,  individuals observed at the leading edge of the spreading
population correspond to lineages whose trajectories remain near the
expanding front at all previous time.  Indeed, let $c^*$ be the spreading speed according to Definition \ref{spread}, then we have
\begin{equation}\label{e.0315.4}
   \{(t,x):~ w=0\} = \{(t,x):~ 0\le x \leq c^* t\}, 
\end{equation}
i.e. the region occupied by the population has the free boundary which is
 given by the ray 
\[
\Gamma_{front}=\{(t,x):\, t\geq 0,\; x=c^* t\}.
\]
Hence, an individual located at the wave front at some time 
\((\bar t,\bar x)=(\bar t,c^*\bar t)\) must have kept pace with the wave front for all earlier times. Motivated by this observation, we introduce the following definition.
\begin{definition}[Local versus Nonlocal Selection of Spreading Speed] \quad 

\begin{itemize}    
\item[{\rm(i)}]    We define a (rightward) spreading speed to be {\it locally selected} (or {\it locally pulled}) if for each $(\bar t, \bar x) \in \Gamma_{front}$, 
    then the corresponding minimizing path $\gamma: [0,t] \to \mathbb{R}$ associated to the control formulation is given by the straight path connecting $(0,0)$ and $(\bar t, \bar x)$, i.e. \eqref{e.0412.1} holds. 

\item[{\rm(ii)}]    A front speed is {\it nonlocally selected} (or {\it nonlocally pulled}) if the minimizing paths associated with points on the front deviate from the front during their evolution.
\end{itemize}
    \label{d.0315.1}
\end{definition}


Solving $L(c^*) = 0$ using \eqref{e.legendre}, we obtain the Freidlin-Gartner formula \cite{gartner1979propagation}:
\begin{equation}\label{e.fg}
c^* = \inf_{\lambda>0} \frac{H(\lambda)}{\lambda}.     
\end{equation}
\begin{corollary}
    In homogeneous or periodic environments such that the Hamiltonian function $\lambda \mapsto H(\lambda)$ is strictly convex, then the spreading speed is always locally selected, and \eqref{e.fg} holds.
\end{corollary}

Note that the above formula has been generalized to periodic environments \cite{weinberger2002spreading, berestycki2008asymptotic,rossi2026freidlingartnerformulaasymptoticprofile} and more general settings \cite{Nadin2020}, where the Hamiltonian (or the energy) $H(\lambda)$ is typically an eigenvalue measuring the linear growth of  modes decaying spatially at rate $\lambda$.

In the present case $H(\lambda) = \lambda^2 + f'(0)$, we have
\begin{equation}\label{e.0319.2}
L(v)=\frac14|v|^2 - f'(0),
\end{equation}
so the representation \eqref{e.lagrangian} 
becomes
\[
w(t,x)
=
\max\left\{
0,\;
\inf_{\substack{\gamma(0)=0\\\gamma(t)=x}}
\int_{0}^{t}
\left(
\frac14|\dot{\gamma}(s)|^2 - f'(0)
\right) ds
\right\}
=\max\left\{0, tL\left(\frac{x}{t}\right)\right\}.\]
and
\[
L(c^*) = 0 \quad \Longleftrightarrow \quad c^* = \inf_{\lambda>0} \frac{H(\lambda)}{\lambda} = 2\sqrt{f'(0)}.
\]

This provides the exact asymptotics for the linear diffusion equation, whose limiting Hamilton-Jacobi equation is $w_t + H(w_x)=0$\footnote{Compare with \eqref{e.hje.1}.} with the unique viscosity solution given by $w(t,x)=tL\left( \frac{x}{t}\right)$. This correspondence demonstrates that the thin-front limit (i.e., the region where 
$w(t,x) >0$) is completely determined by the linear problem 
\begin{equation}\label{e.linear}
    u_t= u_{xx} + f'(0) u,\quad t>0, x\in\mathbb{R}
\end{equation}
with $u(0,\cdot)=u_0:\R\to \R_+$ nontrivial and compactly supported. 

\medskip

\noindent {\bf Pulled Fronts.} For nonlinear systems of equations, we say that a nonlinear front is a {\it pulled front} if its asymptotic speed $c^*$ is governed by the linearized equation around the unstable state
\cite{rothe1981covergence,stokes1976two,garnier2012inside}.
As the above discussion suggests, a broad class of nonlinear systems has {\it pulled fronts}, including \eqref{e.pde1} satisfying the so-called KPP equations \cite{KolmogorovPetrovskyPiskunov1937,lau1985nonlinear,aronson1978multidimensional}:
\begin{align}\label{KPPcond}
   f'(0)>0 \quad \text{ and }\quad f(s) \leq f'(0)s \quad \text{ for }s>0. 
\end{align}
It should be stressed that problems of front propagation into unstable states usually have a continuum of traveling wave solutions, so the above can also be viewed as answering the question of which of these traveling wave solutions are being selected for a large class of initial conditions.

For spreading in general $\mathbb{R}^n$, a natural approach to study front propagation in higher dimension is to rescale and find an effective interface condition or a moving boundary approximation \cite{barles1993front}. In such a case, the propagation of the front is said to be {\it geometric}, i.e. the normal velocity of the front is entirely known based on the location of the interface boundary only. For the equation \eqref{e.pde1} with KPP condition \eqref{KPPcond}, this holds true
when the initial data vanishes outside some bounded zone, in which the normal velocity is given by the minimal speed of the associated traveling wave solutions of the 1-dimensional problem. However, it can be shown that this is not true in general when the initial data has exponential decay at infinity due to the fact that the dynamically important region is heavily influenced by the thin tail of the solution located {\it ahead} of the nonlinear transition associated with the front. One then has to resolve the global problem to determine which wave speed is being selected. Such a reasoning applies to the following cases:
(i) homogeneous environment with noncompactly supported initial data or (ii) 
environments with a shifting habitat.

In fact, if we replace the initial data of \eqref{e.linear} by $u_0(x) = e^{-\lambda x_+}$, where $x_+ = \max\{x,0\}$, then one can show that 
\begin{equation}\label{e.0409.1}
c^* = \begin{cases}
2\sqrt{f'(0)} &\text{ for }\lambda >\lambda_{min},\\
\lambda + \frac{f'(0)}{\lambda} &\text{ for }0<\lambda \le\lambda_{min},
\end{cases}
\end{equation}
where $\lambda_{min} = \sqrt{f'(0)}$. 
This can be understood by considering the heat kernel representation \cite{uchiyama1978behavior,van2003front} of the thin tail of solutions to  \eqref{e.linear} (i.e. {$\lambda>\lambda_{\min}$}),  where 
\[
u(t,x) = \begin{cases}  
    \exp\left( -  \frac{x_+^2}{4t} + tf'(0)   + o(t)\right)  &\text{ for }2\lambda_{\min}<\frac{x}{t} < 2\lambda,\\
    \exp \left( -\lambda x +t(\lambda^2 + f'(0))+o(t) \right) &\text{ for }\frac{x}{t} > 2\lambda.
\end{cases}
\] 
This corresponds to the fact that the unique viscosity solution to the equation 
\begin{equation}
  \begin{cases} 
  J_t + |J_x|^2 + f'(0) = 0 &\text{ in }(0,\infty) \times \mathbb{R},\\
  J(0,x) = \lambda x_+ &\text{ for }x \in \mathbb{R}\end{cases}
\end{equation}
is given by
\begin{equation}\label{e.0315.J}
J(t,x)= \begin{cases}
 \frac{x_+^2}{4t} - tf'(0)   &\text{ for }\frac{x}{t} < 2\lambda,\\
    \lambda x - t(\lambda^2 + f'(0))  &\text{ for }\frac{x}{t} \ge 2\lambda,
\end{cases}
\end{equation}
By setting $J(t,c^*t) = 0$, we obtain the formula \eqref{e.0409.1}.

In a sense, a pulled front spreads in the way the linear spreading point conditions
force it to do. Again, this selection of wave speed by the initial data holds true as well for the nonlinear counterpart \cite{uchiyama1978behavior}, with the unique viscosity solution of
\begin{equation}
  \begin{cases} 
  \min\{w,w_t + |w_x|^2 + f'(0)\} = 0 &\text{ in }(0,\infty) \times \mathbb{R},\\
  w(0,x) = \lambda x_+ &\text{ for }x \in \mathbb{R}\end{cases}
\end{equation} 
being given by $w(t,x) = \max\{J(t,x),0\}$. \blue{For $\lambda\in(0,\sqrt{f'(0)})$, one observes that the speed is given by
\[
c^* = \frac{H(\lambda)}{\lambda}  \qquad \text{ with }\quad  H(\lambda) =\lambda^2 + f'(0),
\]
while 
the control formulation \eqref{e.0725.1} has the optimal path $\gamma(s)$ (the one associated with the endpoint  $(t,c^*t)$) given by
\begin{equation}\label{e.725.2}
\gamma(s)=c^*t-2(t-s)\lambda \quad\text{ for }s\in[0,t].
\end{equation}
Note that the optimal path is straight but does not stay on the invasion front (the path connecting $(0,0)$ and $(t,c^*t)$). 
So 
the speed $c^*$ 
is {\it nonlocally selected} in the sense of Definition \ref{d.0315.1} when the initial data has an exponential tail.}

\section{Alternative Proof of Theorem \ref{t.0312.1}}
\label{s.4}

Next, we illustrate the Hamilton-Jacobi method by outlining an alternative proof of Theorem \ref{t.0312.1} (due to Li et al. \cite{li2014persistence}) concerning the shifting habitat.

\begin{proof}[Second proof of Theorem \ref{t.0312.1} (via the Hamilton-Jacobi method).]
Let $u$ be a solution of the PDE \eqref{e.0312.1} with shifting nonlinearity satisfying \eqref{e.0312.2} and \eqref{e.0312.3}, and such that the initial data has compact support in $\mathbb{R}$. 


We address the case in which the population can keep up with the favorable environment (where $r \approx r(+\infty)$) and persist. Mathematically, this means $2\sqrt{r(+\infty)} > c_1$. In such a case,   \blue{we only need the following weaker version of \eqref{e.persist0723}:}\footnote{The condition \eqref{e.persist} can be compared with \cite[Condition (2.4)]{Nadin2020} concerning propagation success. \blue{See also \cite[Theorem 1.3(iii)]{FangPengZhao2021} for the validation of this condition.}}
\begin{equation}\label{e.persist}
\blue{\liminf_{t \to \infty} \inf_{(c_1+\delta) t \leq x \leq (c_1 + 2\delta) t} u(t,x) >0 \qquad \text{ for each small }\delta>0.} 
\end{equation}


\medskip

\noindent {\bf Step \#1.} 
Since $r$ is bounded, it follows from standard arguments (see, e.g., \cite{Evans1989pde}) that  for each compact subset $K \subset (0,\infty)\times\mathbb{R}$, there exists a constant $C_K>0$ independent of $\epsilon \in (0,1]$ such that
\[
-\ep C_K \leq   w^\ep\left(t,x\right)= -\ep \log u\left(\frac{t}{\ep}, \frac{x}{\ep} \right) \leq C_K \quad  \text{ for }(t,x) \in K.
\]
Hence, the following half-relaxed limits exist and are nonnegative quantities:
\begin{equation}\label{e.halfrelax}
    \overline{w}(t,x) = \limsup_{\ep \to 0^+ \atop (t',x') \to (t,x)} w^\ep(t',x') \qquad \text{ and }\qquad     \underline{w}(t,x) = \liminf_{\ep \to 0^+ \atop (t',x') \to (t,x)} w^\ep(t',x').
\end{equation}
{It is straightforward to observe that }
\begin{equation}\label{e.0315.1}
    \blue{\overline{w}(0,x) =\underline{w}(0,x) = +\infty \quad \text{ for }x > 0.}\end{equation}
Moreover, \eqref{e.persist} implies that, for each $0< \delta \ll 1$, 
\begin{equation}\label{e.0409.2}
    \underline{w}(t,x) = \overline{w}(t,x) = 0 \quad \blue{\text{ for } \frac{x}{t} \in (c_1+\delta,c_1 + 2\delta).}
\end{equation}
We emphasize that the above holds for every $\delta$ sufficiently small.

\medskip

\noindent {\bf Step \#2.} Next, we use the comparison principle to show that the above half-relaxed limits coincide and hence $w^\ep$ converges in $C_{loc}$. This argument follows an idea from \cite{barles1987discontinuous}. Indeed, $\overline{w}$ (resp. $\underline{w}$) is a subsolution (resp. supersolution) of
\begin{equation}\label{e.0315.15}
    \min\{w,w_t + H_1(w_x)\} = 0 \quad \text{ in }\Omega_1:  = \{(t,x):~ t>0, ~ x > c_1t\},
\end{equation}
where 
\(
H_1(p)  =  |p|^2 + r(+\infty).\)

Passing to the half-relaxed limits, it is classical that $\underline{w}$ is a supersolution while $\overline{w}$ is a subsolution \cite{barles1987discontinuous}, and we also have $\underline{w}\leq \overline w$ by construction.

\blue{If we can show $\underline{w} \leq \overline{w}$ by establishing a comparison principle in $\Omega_1:=\{(t,x):~ x > c_1 t,~t>0\}$, then $w^\ep$ converges in $C_{loc}(\Omega_1)$ as $\ep \to 0$, where the limit is the common value $\overline{w} = \underline{w}$ (see Remark \ref{rmk.0725.1}).

To apply the comparison principle, we need the initial data of the super- and subsolutions to be ordered in the correct way, i.e. $\overline{w}\leq \underline{w}$ on the parabolic boundary. However, this is not the case, since $\underline{w}(0,0)=0$ while $\overline{w}(0,0) = +\infty$.}
This issue was resolved in \cite{Evans1989pde} using the control formulation. We present here a PDE-based sliding argument that was used \cite{henderson2026hamilton}.

For each $\eta>0$, let $\overline w^\eta(t,x):= \overline{w}(t+ \eta, x + c_1\eta)$, then one can show that $\overline{w}^\eta(0,x) <\infty$ for all $x \in\mathbb{R}$ and is a viscosity subsolution of \eqref{e.0315.15}. Now, from \eqref{e.0409.2}
\[
\overline{w}^\eta (t,x)=0=\underline{w}(t,x) \quad \text{ pointwise in }\{(t,x):~ x=c_1t,~t\geq 0\} 
\]
and using the fact that $\overline{w}(t,x)$ takes finite value for each $t>0$, we have
\[
\overline{w}^\eta (t,x) < +\infty =  \underline{w}(t,x) \qquad \text{ pointwise in }\{(0,x):~ x>0\}. 
\]
It then follows from standard comparison principle (see, e.g. \cite[Theorem B.1]{Evans1989pde} for Lipschitz continuous solutions or \cite[Theorem A.1]{liu2021asymptotic} for general case) that 
\[
\overline{w}(t+\eta,x + c_1 \eta)= \overline{w}^\eta (t,x)  \leq \underline {w} (t,x) \quad \text{ in }\Omega_1.
\]
Letting $\eta \to 0^+$ and using the fact that $\overline{w}$ is locally Lipschitz continuous in $\Omega_1$, we have\footnote{By construction, we only expect $\overline{w}$ to be upper semi-continuous. Here, actually we used the standard fact that $\overline{w}$ is locally Lipschitz continuous in $\Omega_1$. This is a general property of {\it subsolution} when the Hamiltonian is convex and coercive; See, e.g. \cite[Proposition 1.14]{ishii2013short} or \cite[Lemma 4.5]{henderson2026hamilton}.}
\[
\overline{w}(t,x) \leq \underline{w} (t,x) \quad \text{ in }\Omega_1.
\]
(Even though the inequality fails at $(t,x) = (0,0)$.)
Since by construction, $\overline{w} \geq \underline{w}$ in $\Omega_1$, so we have $\overline{w} = \underline{w}$ in $\Omega_1$. This proves the existence and uniqueness of viscosity solution to \eqref{e.0315.15} verifying the initial-boundary data:
\begin{equation}\label{e.0315.b}
w(0,x)=+\infty \quad \text{ for }x>0, \quad \text{ and }\quad w(t,c_1 t) = 0 \quad \text{ for }t>0.
\end{equation}

\medskip

\noindent {\bf Step \#3.} 
Using $2\sqrt{r(+\infty)} > c_1$, one observes that
\begin{equation}\label{e.0319.1}
w(t,x) = t\max\left\{0, \frac{1}{4}\left|\frac{x}{t}\right|^2 - r(+\infty)  \right\} \quad \text{ for }t>0,~ x > c_1 t,
\end{equation}
is a viscosity solution\footnote{It is a classical solution and hence a viscosity solution when $x \neq 2\sqrt{r(+\infty)}t$, so it suffices to verify Definition \ref{d.0409.1} for the case $x=2\sqrt{r(+\infty)}t$. We leave this verification for the reader.} of \eqref{e.0315.15}. Hence, it must be the unique viscosity solution of \eqref{e.0315.15} verifying \eqref{e.0315.b}. By the definition of $w$ and the fact that $w>0$ in $\{(t,x):~ x > 2\sqrt{r(+\infty)}t\}$, this immediately gives 
\[
\limsup_{t\to\infty} \sup_{x > ct} u(t,x) = 0 \quad \text{ for each }c>2\sqrt{r(+\infty)}. 
\]

\medskip
\blue{
\noindent {\bf Step \#4.} Next, we claim that that for any compact subset $K \subset {\rm Int}\,\{w=0\} = \{(t',x'):~ t'> 0,~ c_1t'<x' < 2\sqrt{r(+\infty)}\}$, we have
\begin{equation}\label{56}
\liminf_{\ep \to 0} \inf_K u\left(\frac{t}{\ep},\frac{x}{\ep} \right) \geq r(+\infty).    
\end{equation}
To prove \eqref{56}, we recall the arguments in \cite[Section 4]{Evans1989pde}.
Let $K$ and $K'$ be compact subsets so that
\[
K \subset \operatorname{Int}K' \subset K' \subset
\{(t',x'):~ t'> 0,~ c_1t'<x' < 2\sqrt{r(+\infty)}\}.
\]
Then  $w^\varepsilon(t,x)\to0$ uniformly in $K'$.
Fix $(t_0,x_0)\in K$ and consider the test function
\[
\phi(t,x)=|x-x_0|^2+(t-t_0)^2.
\]
For all small $\varepsilon$, the function
$w^\varepsilon-\phi$ has a local maximum point
$(t_\varepsilon,x_\varepsilon)$ such that
$(t_\varepsilon,x_\varepsilon)\to(t_0,x_0)$ as
$\varepsilon\to0$.
Furthermore,
$\partial_t\phi(t_\varepsilon,x_\varepsilon),
\partial_x\phi(t_\varepsilon,x_\varepsilon)\to0$,
so that at the point $(t_\varepsilon,x_\varepsilon)$,
\[
o(1)
=\partial_t\phi
-\varepsilon\partial_{xx}\phi
+|\partial_x\phi|^2
\le
\partial_t w^\varepsilon
-\varepsilon\partial_{xx}w^\varepsilon
+|\partial_xw^\varepsilon|^2
\le
u^\varepsilon-r\left(\frac{x_\ep - c_1 t_\ep}{\ep}\right).
\]
This yields
\[
u^\varepsilon(t_\varepsilon,x_\varepsilon)
\ge
r\left(\frac{x_\ep - c_1 t_\ep}{\ep}\right)+o(1) = r(+\infty)+o(1).
\]
In light of
\[
w^\varepsilon(t_\varepsilon,x_\varepsilon)
\ge
(w^\varepsilon-\phi)(t_\varepsilon,x_\varepsilon)
\ge
(w^\varepsilon-\phi)(t_0,x_0)
=
w^\varepsilon(t_0,x_0),
\]
we have
\[
-\varepsilon\log u^\varepsilon(t_\varepsilon,x_\varepsilon)
=
w^\varepsilon(t_\varepsilon,x_\varepsilon)
\ge
w^\varepsilon(t_0,x_0)
=
-\varepsilon\log u^\varepsilon(t_0,x_0),
\]
so that
\[
u^\varepsilon(t_0,x_0)
\ge
u^\varepsilon(t_\varepsilon,x_\varepsilon)
\ge
r(+\infty)+o(1).
\]
Since this argument is uniform for $(t_0,x_0)\in K$,
we deduce \eqref{56}.}

Now, \eqref{56} together with \eqref{e.persist} implies
\[
\liminf_{t\to\infty} \inf_{(c_1+\eta) t< x < (2\sqrt{r(+\infty)}-\eta)t} u(t,x) >0 \quad \text{ for each }\eta\in\left(0, \frac{2\sqrt{r(+\infty)}-c_1}{2}\right). 
\]
\blue{This identifies the asymptotic speed of the leading edge} as $c^* = 2\sqrt{r(+\infty)}$.
\end{proof}
This viewpoint provides an alternative proof of the spreading result obtained
by Bewick, Li, and Fagan \cite{li2014persistence}. Based on the form of $w(t,x)$ given in \eqref{e.0319.1}, it is easy to see that the spreading speed is locally selected according to Definition  \ref{d.0315.1}.

\begin{remark}
    If we consider non-compactly supported initial data satisfying
    \begin{equation}\label{e.bc.0723}
\liminf_{x \to -\infty} u_0(x) >0\quad \text{ and }\quad \frac{1}{x}\log (e^{\lambda x} u_0(x)) \to 0 \quad \text{ as }x \to +\infty,        
    \end{equation}
for some given $\lambda>0$.
Then it is easy to see that $\overline{w}(0,x)=\underline{w}(0,x) = \lambda \max\{x,0\}$. In such a case, the unique viscosity solution is again given by the formula $w(t,x) = \max\{J(t,x),0\}$ with $J(t,x)$ given in \eqref{e.0315.J}. The verification of viscosity solution is left to the reader.
\end{remark}

\section{Nonlocal Pulling in Shifting Habitats with Monotone Environment}\label{s.5}

Next, we consider non-homogeneous environments, i.e. \eqref{e.0312.1} under the KPP condition \eqref{e.0312.2} for some shifting profile $r$. 
We first 
take up the case when $r$ is increasing and bounded from below and above by positive constants and we demonstrate nonlocal selection of spreading speed $c^*$. 

By a suitable rescaling, we may assume without loss of generality that 
 \begin{equation}\label{e.0325.1}
r:\mathbb{R} \to (0,\infty)\quad  \text{ is increasing and }\quad 
r(-\infty) = 1-a  \quad \text{ and }\quad r(+\infty) = 1
\end{equation}
for some $0<a<1$. Observe that, under hyperbolic scaling,
\[
r\left(\frac{x-c_1t}{\ep}\right) \to \begin{cases}
    1 \quad &\text{ for each } (t,x) \text{ such that } \frac{x}{t} > c_1,\\
    1-a \quad &\text{ for each } (t,x) \text{ such that } \frac{x}{t} < c_1.
\end{cases}
\]
In such a case, 
the Hamilton-Jacobi equation is given by \[\min\{w,w_t + H(x/t, w_x)\} = 0\quad \text{ in }\Omega_0 := (0,\infty)\times \mathbb{R}\] with a {\it discontinuous} Hamiltonian
\begin{equation}
 H(x/t,p) =      |p|^2 + 1-a\mathbf{1}_{\{x/t< c_1 \}},
 \end{equation}
with the corresponding Lagrangian being given by
\begin{equation}\label{e.lagrange1}
   L(x/t,v) = \frac{|v|^2}{4} -1 + a   \mathbf{1}_{\{x/t< c_1 \} }.
\end{equation}

For discontinuous Hamiltonian, the classical notion of viscosity solution is due to Ishii \cite{ishii1985hamilton}, involving both the upper and lower envelopes of $H$ (see, e.g., \cite[Definition A.1]{liu2021asymptotic}); See also \cite{lions1987remarks}. 
A self-contained proof of comparison result is presented in \cite[Appendix A]{liu2021asymptotic} by combining the ideas of \cite{ishii1985hamilton,tourin1992comparison}). In case $c_1 < 2\sqrt{r(+\infty)}$, then the limiting rate function $w(t,x)$ coincides with that in \eqref{e.0319.1}. In general, the unique viscosity solution $w(t,x)$ is given by
\begin{equation}\label{e.0315.w}
    w(t,x) = \max\{0, t\hat \rho(x/t)\}.
\end{equation}
Here\footnote{In view of \eqref{e.0315.w},  the second and third cases are only relevant in the definition of $w$ when $c_1 >2$, in which case $\lambda_{nlp}>\sqrt{1-a}$ and $c_{nlp} > 2\sqrt{1-a}$. Note that $\hat\rho = \frac14 s^2 -1$ in case of the linearized equation $u_t - u_{xx} = u$ with $r \equiv 1$.}
\begin{equation}\label{e.0315.hatrho}
   \hat\rho(s) = 
   \begin{cases}
               \frac{1}{4} \left| s\right|^2 - 1 &\text{ for }s \geq  c_1 \\
       \lambda_{nlp}\left(s - c_{nlp} \right) &\text{ for } 2\lambda_{nlp} <s <c_1 \\
               \frac{1}{4} \left| s\right|^2 - 1+a &\text{ for }s \leq   2\lambda_{nlp} , 
   \end{cases}
\end{equation}
where 
\begin{equation}
\lambda_{nlp} = \frac{c_1}{2} - \sqrt{a}\quad \text{  and } \quad c_{nlp} = \lambda_{nlp} + \frac{1-a}{ \lambda_{nlp}}.
\end{equation}
\begin{remark}\label{rmk.0315.2}
\blue{In fact, by the hyperbolic scaling in \eqref{e.halfrelax}, it follows that the limiting rate function is represented by $w(t,x) = t\rho(x/t)$ for some one-variable function $\rho$, and it is standard to deduce a reduced, time-independent equation for $\rho$ (see, e.g.  \cite[Lemma 2.5]{Lam2022asymptotic} or \cite[Lemma 2.2]{liu2021stacked}). 
In fact, $\rho(s)$ is a viscosity supersolution (resp. subsolution) of
\[
\min\{\rho, \rho - s \rho' + |\rho'|^2 + R(s)\}=0  
\]
if and only if $w(t,x) = t\rho(x/t)$ is a viscosity supersolution (resp. subsolution) of
\begin{equation}\label{hje.0722b}
\min\{w, w_t + |w_x|^2 + R(x/t)\}=0 
\end{equation}
In case} 
\[
R(s) = 1-a \mathbf{1}_{\{s < c_1\}}, \qquad \rho(0) = 0,\quad\text{ and }\quad  \lim\limits_{s \to \infty} \rho(s)/s = \infty,
\]
then the unique viscosity solution $\hat\rho$ is given in \eqref{e.0315.hatrho}. Here, the boundary condition at $s=0$ follows from the fact that $\inf r>0$ so that the population persists and spreads at least at the speed $2\sqrt{\inf r}$, while the boundary condition at infinity follows from the fact that the initial data vanishes in $[A,\infty)$ for some $A>0$, and thus
\[
\frac{\hat\rho\left(\frac{1}{t}\right)}{\frac{1}{t}} = w(t,1) \to +\infty \quad \text{ as }t \searrow 0.
\]
\end{remark}


Since the spreading speed is defined via 
\eqref{e.0315.4}, we obtain three regimes for the spreading speed $c^*$ by setting $\hat\rho(s) = 0$ in \eqref{e.0315.hatrho}:
\begin{equation}\label{e.0315.nlp}
    c^* = \begin{cases}
        2 &\text{ when }c_1 \leq 2,\\
        c_{nlp} &\text{ when }2 < c_1 < 2(\sqrt{a} + \sqrt{1-a}),\\
        2\sqrt{1-a} &\text{ when }c_1 \geq 2(\sqrt{a} + \sqrt{1-a}).
    \end{cases}
\end{equation}
In the first and third regimes, the spreading speed is locally selected. In the second regime, however, the path of least action $\gamma \in AC[0, t]$ for a given point at the front is no longer a linear path (see Figure \ref{fig.1}). Indeed, fix a point at the front, i.e. $(t, x) = (t, c_{nlp}t)$, then the trajectory $\gamma$ consists of two linear segments:
\begin{equation}
    \gamma(s) = \begin{cases}
        c_1 s &\text{ for }s \in [0, \tau  t],\,\\
        c_1 \tau t \frac{t-s}{ t(1- \tau)} + x \frac{s-\tau t}{t(1-\tau)} &\text{ for }s \in (\tau t, t],
    \end{cases}
\end{equation}
where $\tau = 1- \frac{c_1 - c_{nlp}}{2\sqrt{a}}$ (note that $\tau \in (0,1)$ in case $2 < c_1 < 2(\sqrt{a} + \sqrt{1-a})$).

\begin{figure}[htbp]
    \centering
    \includegraphics[width=0.55\linewidth]{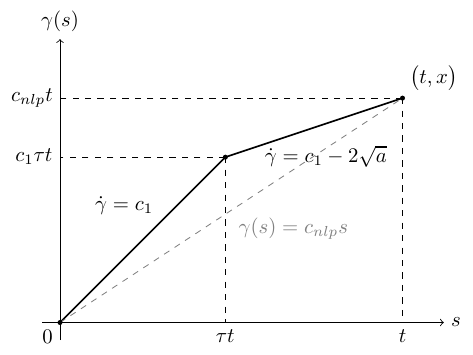}
    \caption{\small A piecewise linear trajectory that first travels ahead of the front $\gamma(s)=c_{nlp}s$ and then relaxes to the front point $(t,c_{nlp}t)$}
    \label{fig.1}
\end{figure}

In this situation, optimal trajectories may travel ahead of the front before
returning, which produces a new mechanism of speed selection known as
\emph{nonlocal pulling}, which was demonstrated first in  \cite{holzer2014accelerated} by considering the linearized problem in the moving frame $c_1$ and then in \cite{girardin2019invasion} when considering the coupled system of two competing species \eqref{e.lv}, where the terminology "nonlocal pulling" was introduced. 

\medskip

\noindent {\bf The case of decreasing environment.} 
For decreasing environmental function (i.e. $r(-\infty) > r(+\infty)>0$), however, our analysis reveals that (see \cite[Theorem 7]{Lam2022asymptotic})
\begin{equation}\label{e.0324.1}
c^* = \begin{cases}
    2\sqrt{r(-\infty)} &\text{ if }c_1 \geq 2\sqrt{r(-\infty)},\\
    c_1 &\text{ if }2\sqrt{r(+\infty)} <c_1 < 2\sqrt{r(-\infty)},\\
    2\sqrt{r(+\infty)} &\text{ if }c_1 \leq 2\sqrt{r(+\infty)}.
\end{cases}
\end{equation}
Hence, the speed is either locally selected (cases (i) and (iii)), or the population invades with the same pace as the shifting environment (the forced wave regime). In all cases, for compactly supported initial data, the spreading speed is always locally selected. This is the case for compactly supported initial data.

\blue{A recent study on the nonlocal diffusion problem by Tao et al. \cite{tao2026complete} revealed a novel nonlocal speed enhancement mechanism. 
As we have seen before, in homogeneous environment the speed $c^* = H(\lambda)/\lambda$ is nonlocally selected for noncompactly supported, exponentially decaying initial data. Here, the optimal path $\gamma$ is straight (thanks to Jensen's inequality) but it does not initiate at the origin (see \eqref{e.725.2}). What they found is that there is a regime in which this path $\gamma$ can sometimes be ``bent" by the decreasing, shifting environment. This happens when the initial data exhibit exponential decay within an appropriate range of decay rates relative to the environmental speed $c_1$. The following result, inspired by \cite{tao2026complete}, demonstrates that the same underlying mechanism also applies to reaction-diffusion equations. 

\medskip

\begin{theorem} (Nonlocal pulling in decreasing environment) 
    Let $u$ be a solution of the PDE \eqref{e.0312.1} with shifting nonlinearity satisfying \eqref{e.0312.2} and  that $r$ is non-increasing with 
$r(-\infty) > r(+\infty) >0$ and has exponentially decaying initial data \eqref{e.bc.0723}. 
    Assume $\lambda \in \big(0,\sqrt{r(-\infty)}\,\big)$ and $c_1 >0$ satisfies
\begin{equation}\label{e.0723.1}
    c_1 > \max\left\{ \frac{H^+(\lambda)}{\lambda}, \frac{H^-(\sqrt{r(-\infty)}) - H^+(\lambda)}{\sqrt{r(-\infty)}-\lambda} \right\}
\end{equation}
where $H^\pm(p) = p^2 + r(\pm \infty)$. 
Then the spreading speed is given by 
\[
c^* = \frac{H^-(\hat\lambda)}{\hat\lambda}
\]
where $\hat\lambda \in \big(0,\sqrt{r(-\infty)}\,\big)$ is the unique root of
\[
\hat\lambda c_1 - H^-(\hat\lambda) = \lambda c_1 - H^+(\lambda).
\]
Furthermore, we observe the nonlocal pulling phenomena when the invasion wave is behind the environmental shift and yet faster than the local KPP speed. In fact, we have 
\[
2\sqrt{r(-\infty)} < c^* < \min\left\{ \frac{H^-(\lambda)}{\lambda} , c_1\right\},
\]
where $2\sqrt{r(-\infty)}$ is the linearly selected speed, $\frac{H^-(\lambda)}{\lambda}$ is the nonlocally selected speed in the homogeneous environment with Hamiltonian $H^-(p)$ and $c_1$ is the speed of environmental shift.
\end{theorem}
\begin{proof}[Sketch of proof.] Since it is known that the rate function $w^\ep$ converges to the unique viscosity solution of 
\begin{equation}\label{hje.0724}
\begin{cases}
\min\left\{w,w_t + |w_x|^2 + r(-\infty)\mathbf{1}_{\{x/t\leq c_1\}} + r(+\infty) \mathbf{1}_{\{x/t> c_1\}}\right\}=0  \quad \text{ for }t>0,x \in \mathbb{R} \\
w_0(x) = \lambda \max\{x,0\},
\end{cases}
\end{equation}
and that the population persists in the region ${\rm Int}\,\{(t,x):~w(t,x) = 0\}$ and is by definition exponentially small in $\{(t,x):~w(t,x) > 0\}$, it suffices to determine the limit rate function $w(t,x)$. By uniqueness, we only need to produce an explicit solution of \eqref{hje.0724}.

\noindent {\bf Step \#1.} Observe that \eqref{e.0723.1} is equivalent to 
\[
0<\rho(c_1)< c_1\sqrt{r(-\infty)}  - H^-(\sqrt{r(-\infty)}).\]

\medskip 

\noindent {\bf Step \#2.} Observe that $\hat\lambda \in (0,\sqrt{r(-\infty)})$ is well defined if \eqref{e.0723.1} holds.

\medskip 

\noindent {\bf Step \#3.} Define  
\[
\rho(s) = \begin{cases}
    \lambda s - H^+(\lambda) &\text{ for }s \geq c_1,\\
    \hat\lambda s - H^-(\hat\lambda) &\text{ for }s < c_1,
\end{cases}
\]
Then $\max\{0,\rho(s)\}$ is the (unique) viscosity solution of
\begin{equation}\label{hje.0724b}
\begin{cases}
\min\left\{\rho,\rho - s\rho' + |\rho'|^2+ r(-\infty)\mathbf{1}_{\{s\leq c_1\}} + r(+\infty) \mathbf{1}_{\{s> c_1\}}\right\}=0  \quad \text{ for }s>0\\
\rho(0)=0, \quad \lim_{s \to \infty}\rho(s)/s = \lambda. 
\end{cases}
\end{equation}
This implies that \(w(t,x) = \max\{0,t\rho(x/t)\}
\) is the unique viscosity solution of \eqref{hje.0724}. From the characterization $\rho(c^*) =0$, we deduce that $c^* = H^-(\hat\lambda)/\hat\lambda$. 
\end{proof}

}

\bigskip

\noindent {\bf Nonlocal selection of speed arising in spreading in higher dimensional, spatially heterogeneous environment.} 
 An interesting example is given in \cite[Propositions 3.13 and 3.14]{Nadin2020}. 
Consider the Fisher-KPP equation on $\mathbb R^2$:
\begin{equation}
    u_t -  \Delta u = a(x_1)u(1-u) \quad \text{ for }t>0,~x=(x_1, x_2) \in \mathbb{R}^2,
\end{equation}
where $u(0,x)=u_0(x)$ is compactly supported and $a:\mathbb{R} \to \mathbb{R}$ is a monotone increasing function such that there exist $0 < a_-<a_+$ such that 
\[
a_- = \lim_{x_1 \to -\infty} a(x_1) \quad \text{  and } \quad a_+= \lim_{x_1 \to +\infty} a(x_1).\] 
Then the limiting rate function $w(t,x)$ is the unique viscosity solution satisfying
\[
\min\{w, w_t + |\nabla_x w|^2 + a_{-}\mathbf{1}_{\{x_1<0\}} + a_+\mathbf{1}_{\{x_1\geq 0\}}\} = 0, \quad t>0, x\in\R^2,
\]
\[
w(0,x_1,x_2)=\begin{cases} 0  \quad &\text{ if }(x_1,x_2)= (0,0),\\
+\infty &\text{ if }(x_1,x_2)\neq (0,0).\end{cases}
\]
Moreover, we have
\[
\{(x_1,x_2):~ w(t,x_1,x_2) = 0\} = t\mathcal{S} \quad \text{ for each }t>0,
\]
with $\mathcal{S}$ being the {\it convex envelope} of
\[
\{x \in \mathbb{R}^2:~ |x|\leq 2\sqrt{a_+},~ x_1 \geq 0\} \cup 
\{x \in \mathbb{R}^2:~ |x|\leq 2\sqrt{a_-},~ x_1 \leq 0\}.
\]
Then for each ${\bf e} \in \mathbb{S}^1$, the directional speed is given by
\[
v({\bf e}):= \sup\{s \geq 0:~ s{\bf e} \in \mathcal{S}\}.
\]
For each direction ${\bf e}\in \mathbb{S}^1$ such that $v({\bf e})  \in (2\sqrt{a_-},2\sqrt{a_+})$, it follows that the speed is nonlocally selected. Heuristically,  to reach as far
as possible in these directions during a given time $t$, an individual has to first go in direction $(0,1)$ at
speed $2\sqrt{a_+}$ and then to get into the left medium at speed $2\sqrt{a_-}$. This change in the direction of the optimal path verifies that these directional speeds are nonlocally selected.

\medskip

\noindent {\bf Generality.} 
By pushing the Hamilton-Jacobi approach further, we establish a general result in \cite{Lam2022asymptotic} for Nicholson’s blowfly model with distributed delay in the birth rate \cite{so2001reaction}. In the case of the reaction-diffusion equation, we considered the Fisher-KPP equation with general spatio-temporal heterogeneous growth rate 
\begin{equation}\label{e.kpp.gen}
    u_t - u_{xx} =  (r(t,x) - u)u.
\end{equation}
To state our assumption in $r(t,x)$, we give some definitions.
\begin{definition}\label{d.1}
    Consider the following half-relaxed limits:
\[
    R(s)=\limsup_{t\to\infty \atop s' \to s} r(t,s't) \qquad \text{ and }\qquad \underline{R}(s) = \liminf_{t\to\infty \atop s' \to s} r(t,s't) 
    \]
\end{definition}
\begin{definition}
We say that $h:\mathbb{R}\to\mathbb{R}$ is \emph{locally monotone} if, for each $s_0$, either
\[
\lim_{\delta\to 0}\,
\inf_{\substack{|s_i-s_0|<\delta\\ s_1<s_2}}
\bigl(h(s_1)-h(s_2)\bigr)\ge 0,
\qquad \text{ 
or }\qquad 
\lim_{\delta\to 0}\,
\sup_{\substack{|s_i-s_0|<\delta\\ s_1<s_2}}
\bigl(h(s_1)-h(s_2)\bigr)\le 0.
\]
\end{definition}
\begin{remark}
Equivalently, $h$ is locally monotone if for any $s \in \mathbb{R}$, the one-sided limits $h(s^\pm)$ exist and that $h(s)$ is bounded between $h(s^+)$ and $h(s^-)$.
\end{remark}
By establishing the necessary comparison principle, we obtained the existence and characterization of spreading speed under the following assumptions:
\begin{description}
 \item[(H1)] $R(s)$ is locally    monotone. 
    \item[(H2)] The functions  $\underline{R}(s)$ 
    satisfy
    \[
    \underline{R}(s) = {R}(s) \quad \text{ a.e. in }(0,\infty),\quad \text{ and }\quad \inf_{s\in(0,\infty)}\underline{R}(s) >0.
    \]

\end{description}
\begin{theorem}[\cite{Lam2022asymptotic}]
Let $u$ be a solution to \eqref{e.kpp.gen}  with compactly supported initial data and assume that the half-relaxed limits of $r(t,x)$ satisfy conditions \textup{\bf(H1)-(H2)}. 
Then
as $\ep \to 0$, $w^\ep(t,x)$ converges to the limit $t \hat\rho(x/t)$ in $C_{loc}((0,\infty)\times \mathbb{R})$ where $\hat\rho(s)$ is the unique viscosity solution\footnote{Here we mean viscosity solution in the sense of Ishii \cite{ishii1985hamilton}.} of
\begin{equation}\label{hje.0722}
\begin{cases}
\min\{ \rho, \rho - s \rho' + |\rho'|^2 + R(s) \}= 0 \quad \text{ for }s \in (0,\infty),\\
    \rho(0) = 0, \quad \lim\limits_{s \to \infty} \rho(s)/s = \infty.
\end{cases}
\end{equation}
Moreover, there exists $c^*>0$ such that \[\{s\geq0:~\hat \rho(s)=0\} = [0,c^*],\] and we have
\begin{equation}\label{e.0316.1}
\lim_{\ep \to 0} u^\ep =
\begin{cases}
    0 & \text{ uniformly on compact subsets of }\{(t,x):~ x > c^* t\},\\
    \geq \inf \underline{R} & \text{ uniformly on compact subsets of }\{(t,x):~ 0\leq x < c^* t\}.
\end{cases}
\end{equation}
\end{theorem}
\begin{remark}
\blue{
The assumption that $\inf \underline{R}>0$ can be relaxed. Indeed, we can go through Step \#4 in the proof contained in Section \ref{s.4} and show that 
\[
\lim_{\ep \to 0} u^\ep =0  \quad \text{ uniformly on compact subsets of }\{(t,x):~ x > c^* t\}\]
and for each $\delta>0$, we have
\[
\liminf_{\ep \to 0\atop (t',x')\to (t,x) } u^\ep(t',x')  
\geq \delta,
\]
uniformly for $(t,x)$ on compact subsets of $\{(t,x):~ 0\leq x < c^* t,~ \underline{R}(\tfrac{x}{t}) \geq \delta\}$.
In particular, 
\[
\lim_{n \to \infty} u(t_n,x_n) >0
\]
for any sequence $(t_n,x_n)$ such that $t_n \to \infty$ and $ \underline{R}(\lim\limits_{n\to\infty}\frac{x_n}{t_n})>0$.} 
\end{remark}

We present several classes of $r(t,x)$ which can be handled by the theoretical results in \cite{Lam2022asymptotic}.

\medskip

\noindent {\bf Example 1.} 
    (Non-monotone environments) This corresponds to $r(t,x) = r(x-c_1t)$ such that $r(\pm \infty)$ exist. Then the monotonicity requirement can be relaxed to
    \[
    \sup_{\mathbb{R}} r \leq \max\{ r(\pm \infty)\},
    \]
    and one can derive the exact same Hamilton-Jacobi equation. 
 In particular, when  $r(-\infty) =1-a$ and $r(+\infty)=1$, then the spreading speed is still given by \eqref{e.0315.nlp}.
\medskip

\noindent {\bf Example 2.} 
    (Asymptotically homogeneous environments) This corresponds to $r(t,x)$ for which there is a constant $r_0>0$ such that the half-relaxed limits $R(s)$ and $\underline{R}(s)$ in Definition \ref{d.1} satisfy \[R(s) = r_0 \quad \text{ for all }s\geq 0 \quad\text{ and } \underline{R}(s) = r_0\quad \text{ for almost every }s \geq 0,\] in which case 
\[
c^* = 2\sqrt{r_0}.
\]
\medskip

\noindent {\bf Example 3.} 
Consider the class of environments where \[
r(t,x) = \sum_{j=1}^m r_j(x - c_j t)
\]
where for each $j$,  $r_j$ is a monotone (increasing or decreasing), bounded, positive function. If all $c_j$ are distinct, then the resulting $R(s)$ will be piecewise constant and satisfies {\bf (H1)} and {\bf (H2)}.
See also \cite{Lam2022asymptotic,YiZou2025} for a Nicholson blowfly equation with multiple shifting speeds. In general, we can allow any continuous or monotone $R(s)$ in our analysis.

\medskip

\noindent {\bf Example 4.}
Let $R_1: \mathbb{R} \to (0,\infty)$ be an upper semicontinuous function such that $\inf R_1>\delta$ (for some $\delta>0$) and  {\bf (H1)} holds, then let $r(t,x)$ be an environmental function satisfying
\[
r(t,x) = R_1(x/t) - h(x/t) \quad \text{ for }t>0,
\]
such that \(0 \leq h(s) \leq \delta \) and $h$ is supported on a set of measure zero in $\mathbb{R}$. 
Then {\bf (H1)-(H2)} are verified. 
\medskip

\begin{remark}
    The viscosity solution framework is extremely flexible and is amenable to the treatment of different kinds of propagation dynamics, for which we focus on some of the simplest examples in this article. Besides  equations with  temporal delay in \cite{Lam2022asymptotic},  see \cite{Nadin2020} in relation to homogenizations for  reaction-diffusion equations in $\mathbb{R}^n$;   \cite{tao2026complete} for nonlocal diffusion on $\mathbb{R}$; and \cite{liang2020spreading} for diffusion on lattices.
\end{remark}

\begin{problem}
Suppose $r(t,x)$ satisfies {\bf(H1)-(H2)} and assume in addition that $r_+:= r(t,+\infty)$ (resp. $r_- = r(t,-\infty)$) exists for each $t$ and is a positive constant independent of $t$. Characterize the set of entire solutions $U(t,x)$ such that  
\[
\lim_{x\to - \infty} U(t,x) = r_- \quad \text{ and }\quad \lim_{x \to +\infty} U(t,x) =0  \qquad \text{ for all }t \in \mathbb{R}.
\]
\end{problem}

\begin{problem}
Consider the heterogeneous Fisher-KPP problem \eqref{e.kpp.gen} with $r(t,x)$ satisfying
\[
\begin{cases}
g_1(x-c_1t)\leq r(t,x) \leq g_2(x-c_1 t)\\
\lim\limits_{x \to -\infty} |r(t,x) - g_1(x-c_1t)| = 0 =\lim\limits_{x \to +\infty} |r(t,x) - g_2(x-c_1t)|,  
\end{cases}
\]
where $g_i$ are positive, periodic functions. Derive the limiting Hamilton-Jacobi equation and characterize the spreading speed.
\end{problem}

\noindent {\bf Related Literature.} 
The Hamilton–Jacobi approach to propagation problems was originally developed by Freidlin and G\"{a}rtner \cite{gartner1979propagation,Freidlin1985limit,freidlin1986geometric} using probabilistic large deviation techniques. It was later reformulated using PDE arguments based on the concept of viscosity solutions by Evans and Souganidis \cite{Evans1989pde}. This methodology has since been extended to treat more general settings, including equations with periodic space–time coefficients \cite{majda1994large} as well as stochastic ergodic environments \cite{souganidis1999stochastic,lions2010stochastic}. More recently, this approach has been applied to age-structured population model \cite{kang2025global}, neural field equation 
\cite{tao2025hamilton}, 
the reactive-telegraph equation 
\cite{henderson2017reactive}, 
 integro-differential equations \cite{bouin2018thin,brandle2009large,brandle2013large,perthame2005front}, 
problems with fractional laplacian \cite{meleard2015singular,souganidis2019front}, and the
 cane toad equations \cite{bouin2017super,calvez2022non}.











\section{Nonlocal Pulling in  Shifting Habitats with Non-Monotone Environment}\label{s.6}

Consider the Fisher--KPP equation in a shifting heterogeneous environment given in \eqref{e.0312.2}, in which the growth rate of the population is given by $r(x-c_1t)$, 
where $c_1$ represents the speed of the shifting habitat and $r$ is a bounded
positive function describing the spatial heterogeneity of the environment.


In view of the conditions {\bf(H1)-(H2)}, the Hamilton--Jacobi approach of \cite{Lam2022asymptotic} could be applied under the
assumption
\begin{equation}\label{e.0315.5}
    0 < \inf_{\mathbb{R}} r \leq \sup_{\mathbb{R}} r \leq \max\{r(-\infty),r(+\infty)\}.
    \end{equation}
While the first inequality $\inf_{\mathbb{R}} r>0$ can sometimes be relaxed to $\inf_{x > c_0 t} r>0$ for some $c_0>0$  as presented in Section \ref{s.4}, the last inequality 
\(\sup_{\mathbb{R}} r \leq \max\{r(-\infty),r(+\infty)\}\) is crucial for uniqueness of the limiting Hamilton-Jacobi equation. 

Indeed, when this condition
fails (i.e. \(\sup_{\mathbb{R}} r >\max\{r(-\infty),r(+\infty)\}\)), the invasion may be influenced by localized regions of higher growth,
and the corresponding Hamilton--Jacobi problem may admit multiple classical viscosity
solutions in the sense of Ishii  \cite{barles2024modern,giga2013hamilton}.

In this section, we discuss the treatment of this general case in \cite{lam2025asymptotic} based on a different notion of viscosity solutions. 
An alternative treatment, based on the direct construction of super/subsolutions to the parabolic problem, can be found in \cite{GirardinGilettiMatano2024} .

The key contribution of \cite{lam2025asymptotic} is to derive a Hamilton-Jacobi equation with an additional junction condition, and show that
the limiting rate function $w$ can be uniquely identified by the notion of \emph{flux-limited solutions} of such HJE with a junction condition, following the
framework of Imbert and Monneau for Hamilton--Jacobi equations with junction
conditions \cite{Imbert2017flux,Imbert2017quasi}.  In this setting, the thin-front limit of the rescaled rate function
\[
w_\varepsilon(t,x)=-\varepsilon \log u_\varepsilon(t,x),
\qquad
u_\varepsilon(t,x)=u\!\left(\frac{t}{\varepsilon},\frac{x}{\varepsilon}\right),
\]
is shown to converge to a function of the form\footnote{See Remark \ref{rmk.0315.2}.}
\[
w(t,x)=t\,\hat{\rho}(x/t),
\]
where $\hat{\rho}$ solves a Hamilton--Jacobi equation with an extra junction condition
at the environmental speed $s=c_1$. More precisely, the limiting profile $\hat\rho$ satisfies
\begin{equation}\label{e.hje.fl}
\begin{cases}
\min\{\rho,\rho - s\rho' +H(s,\rho')\}=0 \qquad  \text{ for }s \in (0,c_1) \cup (c_1,\infty),\\[2mm]
\min\{\rho(c_1),\rho(c_1)+\max\{A,F(\rho'(c_1+),\rho'(c_1-))\}\}=0,
\end{cases}
\end{equation}
where $H(s,p) =|p|^2 + R(s)$ is as before, $F(p,\tilde p)$ is nonincreasing in $p$ and nondecreasing in $\tilde p$. Particularly, the junction condition\footnote{The function $F$ in the junction condition depends on the increasing and decreasing part of the Hamiltonian $H$, as well as the flux limiter $A$. See \cite{lam2025asymptotic} for details.} involves a \emph{flux-limiter} incorporating the
influence of the spatial heterogeneity on the spreading dynamics:
\[
\text{(Flux-limiter)}  \qquad \qquad \qquad 
A=\Lambda_1-\frac{c_1^2}{4},\qquad \qquad \qquad \qquad 
\]
and $\Lambda_1$ denotes the generalized principal eigenvalue associated with
the environmental profile $r$:
\[
\Phi''(y) + r(y) \Phi(y) = \Lambda_1 \Phi(y) \quad \text{ for }y \in \mathbb{R}
\]
which is defined by\footnote{This and several other notions of principal eigenvalues are analyzed in \cite{berestycki2015generalizations}.} 
\[
\Lambda_1 := \Lambda_1(r)
= \inf \left\{ \Lambda \in \mathbb{R} :
\exists\, \phi \in C^2_{\mathrm{loc}}(\mathbb{R}),\ \phi>0,\ 
\phi'' + r(y)\phi \le \Lambda \phi \ \text{in } \mathbb{R}
\right\}.
\]
The major difference with viscosity solution is the use of test function that is only piecewise $C^1$ across the junction $s=c_1$:
\[
C^1_{pw}=\{\psi \in C((0,\infty)):~  C^1((0,c_1]) \cap C^1([c_1,\infty)).
\]
so that at the junction, the test function in the verification of junction condition can acquire different directional derivatives at $s=c_1$ as appear in the equation \eqref{e.hje.fl}.

\blue{The flux-limited viscosity solution admits an optimal control interpretation as
\[
w(t,x)=\max\{0,\inf_{\substack{\gamma(0)=0\\ \gamma(t)=x}}
\int_0^t \mathcal{L}\bigl(s,\gamma(s),\dot\gamma(s)\bigr)\,ds\},
\]
where the Lagrangian depends on the location of the trajectory. Away from the
interface, one has the standard running cost
\[
\mathcal{L}(t,x,v)=\frac{|v|^2}{4}-r,
\]
with $r=1$ in the low-resistance region and $r=1-a$ in the high-resistance
region. The essential new feature introduced by the interface is that the environment
is discontinuous. Consequently, trajectories crossing or moving along the
interface cannot be described solely by the bulk Lagrangians. Instead, the
interface carries a reduction in cost determined by the flux limiter. In the
present problem, this interface contribution is represented by the constant
$\Lambda_1$, yielding the effective Lagrangian 
\[
\mathcal{L}(t,x,v)
=
\frac{|v|^2}{4}
-
\Bigl[
(1-a)\mathbf{1}_{\{x/t<c_1\}}
+
\Lambda_1\mathbf{1}_{\{x/t=c_1\}}
+
\mathbf{1}_{\{x/t>c_1\}}
\Bigr].
\]
Thus, the interface acts as a distinct propagation medium whose effective cost
is determined by the flux limiter $\Lambda_1-\frac{c^2_1}{4}$.
Note that when $\Lambda_1 \leq 1$, then the effective Lagrangian reduce to the original case shown in \eqref{e.lagrange1}.
}

As before, the spreading speed is characterized by the free boundary point of the
rate function,
\[
c^*=\sup\{s\ge0:\hat{\rho}(s)=0\}.
\]
Thus the quantities $r(+\infty)$, $r(-\infty)$, $c_1$, and the eigenvalue
$\Lambda_1$ together determine the invasion speed. The following result is to be compared with \eqref{e.0315.nlp}.

\begin{theorem}[{\cite[Theorem 3.2]{lam2025asymptotic}}]
Let $u$ be a solution of \eqref{e.0312.1} with compactly supported, nonnegative,
nontrivial initial data.
\blue{Assume \eqref{e.0312.2} holds for some bounded function \(r:\mathbb{R} \to (0,\infty)\) such that $0<\min\{r(-\infty),r(+\infty)\}$. }
Then the rate function $w^\ep$ converges in $C_{loc}((0,\infty)\times [0,\infty))$ to $t \hat\rho(x/t)$, where $\hat\rho$ is the unique {\it flux-limited} solution to \eqref{e.hje.fl}. Moreover, 
\[
\{s \geq 0:~ \hat\rho(s) = 0\} = [0,c^*],
\]
so that $c^*$ is the spreading speed, i.e. \eqref{e.0316.1} holds.

Furthermore, the free boundary point $c^* = c^*(r(\pm \infty),\Lambda_1)>0$ is given by (denoting $r(\pm \infty) = r_{\pm}$)
\[
c^*=
\begin{cases}
2\sqrt{r_+}, 
& \text{if } c_1 \le 2\sqrt{r_+},\\[6pt]

c_1, 
& \text{if } 2\sqrt{r_+} < c_1 \le 2\sqrt{\Lambda_1},\\[6pt]

\dfrac{c_1}{2}-\sqrt{\Lambda_1-r_-}
+\dfrac{r_-}{\dfrac{c_1}{2}-\sqrt{\Lambda_1-r_-}},
& \text{if } 2\sqrt{\Lambda_1} < c_1 \le
2\bigl(\sqrt{r_-}+\sqrt{\Lambda_1-r_-}\bigr),\\[10pt]

2\sqrt{r_-},
& \text{if } c_1 >
2\bigl(\sqrt{r_-}+\sqrt{\Lambda_1-r_-}\bigr).
\end{cases}
\]
\label{thm6.1}
\end{theorem}

\begin{remark}
(i) The invasion front is nonlocally selected only when \[2\sqrt{\Lambda_1} < c_1 < 2(\sqrt{r_-} + \sqrt{\Lambda_1 - \blue{r_-}})\] and is otherwise locally selected. 

\noindent (ii) The invasion front is a pulled front, except in the case \[2\sqrt{r_+} < c_1 < 2\sqrt{\Lambda_1}\] in which it \red{contains} a pushed front. 
\end{remark}

\blue{A key outcome} of this work is that it unifies several previously
observed spreading phenomena within a single Hamilton--Jacobi framework.
Depending on the relation between these parameters, one may observe
\begin{itemize}
\item classical KPP spreading with a locally selected speed $2\sqrt{r_{\pm}}$,
\item nonlocally pulled fronts influenced by distant favorable habitats,
\item or forced traveling waves whose speed coincides with the habitat
shift $c_1$.
\end{itemize}

Another important feature of the approach is that it naturally extends to
environments with multiple shifting speeds. In that case, the limiting
Hamilton--Jacobi equation contains multiple junction conditions, each
associated with a different environmental interface.

These results provide a unified description of invasion dynamics in
heterogeneous shifting environments and clarify the transition between
locally pulled fronts, nonlocal pulling, and forced wave propagation.

When $c^*=c_1$, the existence and global attractivity of forced wave is established by  Berestycki and Fang {\cite[Theorem 1.5(iii)]{berestycki2018forced}}:
\begin{theorem}
When $2\sqrt{r_+} < c_1 < 2\sqrt{\Lambda_1}$, there exists a minimal forced traveling wave solution and that it attracts all solutions with compactly supported initial data. In particular, $c^* = c_1$.
\end{theorem}
In fact, they show that the minimal forced traveling wave attracts all initial data of the form $e^{-\lambda x_+}$, with $\lambda > \frac{c_1 - \sqrt{c_1^2 - 4r_+}}{2}$.

\section{Other Methods to Study Propagation Dynamics in Shifting Habitats}

Li and collaborators first investigated the problem of propagation in two connected moving semi-infinite habitats with different levels of
quality for reaction-diffusion equation \eqref{e.0312.1} in   \cite{li2014persistence} and for  integro-difference equations in \cite{li2016persistence,li2020traveling}. They established the existence of spreading speeds and traveling waves. Their approach is based on the discrete-time monotone semiflow framework developed by Weinberger \cite{weinberger1982long}, in which suitable subsolutions are constructed and propagated through iteration arguments to characterize the spreading dynamics.  
Such ideas are also applied to a reaction–diffusion model \cite{hu2020spreading} and to the nonlocal-diffusion model \cite{li2018spatial}. 

This framework was generalized to a functional analytic setting by Yi and Zhao in \cite{YiZhao2020}. 
Specifically, the authors established a general dynamical systems framework, in the flavor of \cite{liang2027asymptotic}, where the shifting environment is conceptualized as  environments without translational invariance but for which the right and left hand spatial limits exist and are translational invariant. 
As such, their framework can be applied to equations or systems incorporating such feature as time-delay or propagation on cylindrical domains. Subsequent works \cite{YiZhao2023, YiZhao2025JDE, YiZhao2025} further developed this framework by relaxing the translation monotonicity assumption. In particular, Yi and Zhao \cite{YiZhao2023} refined the analysis by introducing the notion of asymptotic annihilation and establishing propagation results under asymptotic translation invariance.
In \cite{YiZhao2025JDE}, propagation dynamics were studied for non-monotone evolution systems with two limiting systems admitting translation invariance and spreading speeds. As for the question of spreading speed, they are mainly interested in the case when the population can surpass the shifting habitat (i.e. $c^* > c_1$, from condition (UC) therein). In such a case, the spreading speed is equal to that of the limiting homogeneous system, and is therefore locally selected. This aspect of their result can be compared with \cite{hu2020spreading}. Here, we include some more recent applications:

\begin{enumerate}
    \item Reaction-diffusion equations in time-heterogeneous shifting habitats \cite{FangPengZhao2021, ZhangZhao2025} 
    \item Systems in time-periodic shifting habitats \cite{HouWangZhaoJDE2024,HouWangZhao2025JNS, HouWangZhaoJDE2025} 
    \item Shifting habitat with time-delay effects \cite{YiZou2025,JiangZhao2026}
    \item  Integro-difference models,  impulsive systems with birth pulse dynamics \cite{JiangYiZhao2024, ZhangYiWu2024, ZhangYiChen2024, LuWangZhao2026}.
    \item System in a cylinder with shifting effect \cite{GuoYiZhangZou2026}
\end{enumerate}


Apart from the Hamilton-Jacobi method mentioned in earlier sections, spatial invasion involving nonlocal pulling can also be handled with the construction of super/subsolutions. See, e.g. \cite[Chapter 6]{lam2022introduction} where formula \eqref{e.0315.nlp} was proved using elementary arguments. For the problem involving a shifting diffusivity; See \cite{faye2022asymptotic}.
More recently, a refined construction was introduced to estimate the refined logarithmic delay (see Section \ref{s.11} and \cite{LamWu_preprint}). Super/subsolutions method was also used to study \blue{competitive systems in a road-field environment with climate change \cite{Ou2025}; nonlocal dispersal models  \cite{VoTaVuDo2026, LiuLiXu2026};  predator-prey systems} with shifting habitat \cite{ZhouWuJDE2025};
force waves for a population model with mobile and stationary
states \cite{XuXiaoWu2025};  
spreading for noncooperative system with time delay in a shifting habitat \cite{XieLin2025}; the
Lotka-Volterra system in time-period habitat \cite{WangQiao2025}; 
a lattice system in spatially period and shifting habitat \cite{SanWangWu2026}. 

\section{Spreading in Competition Systems}\label{spread-competition}


Returning to the discussion of the full competition system \eqref{e.lv}, one observes that the scalar
shifting-habitat model captures essential features of multi-species invasion.

Motivated by the post-glacial northward migration of several tree species
into newly deglaciated regions of North America, Shigesada et al.~\cite[Ch.~7]{shigesada1997biological}
formulated the problem of spatial spread for two or more competing species
in an initially unoccupied habitat. For two competing species, it was
conjectured that, for large times, the solution resembles a pair of stacked
traveling fronts, that is, two transition layers propagating at distinct
speeds $c_1>c_2$ and connecting three homogeneous equilibrium states
$(0,0)$, $E_1$, and $E_2$. Here $E_1$ denotes the semi-trivial equilibrium
in which the faster species persists, while $E_2$ represents either the
other semi-trivial equilibrium or the coexistence equilibrium. Although
the spreading speed $c_1$ of the faster species can typically be predicted
from the corresponding single-species equation (since the slower species
is negligible near the leading edge), determining the second speed $c_2$
remained unresolved for many years.

Lin and Li~\cite{lin2012asymptotic} first studied the spreading behavior of the competition system
\eqref{e.lv} in the weak competition regime $0<a<b<1$ in which every component of the initial data is compactly
supported, obtaining bounds for the spreading
speed $c_2$ of the slower species. In the strong competition case
$a,b>1$, Carr\`ere~\cite{carrere2018spreading} determined both spreading speeds,
showing that $c_2$ coincides with the unique wave speed of traveling
fronts connecting the semi-trivial equilibria $(1,0)$ and $(0,1)$;
see also \cite{peng2021sharp,wu2023sharp} for more refined asymptotic descriptions.

For system \eqref{e.lv} in the monostable case when $0<a<1$, we observe {\it nonlocal selection} of spreading speed, i.e. the second spreading speed $c_2$ may
depend on the speed $c_1$ of the leading front. This phenomenon was
demonstrated by Holzer and Scheel~\cite{holzer2014accelerated} in the study of 
the ``decoupled'' case $a=0$ and $b>0$. Their analysis showed that $c_2$
is determined by the linear instability of the zero state for an elliptic
problem with spatially heterogeneous coefficients.

In the coupled case, the intermediate regime $0<a<1<b$ was studied in
a recent work of Girardin and the first author~\cite{girardin2019invasion}. In case the weak competitor $v$ is the faster spreader (i.e. $dr>1$)  an explicit expression for $c_2$ is derived, namely,
\begin{equation}\label{e.lv.nlp}
c_2 = \max\{c^*,c_{LLW}\},     
\end{equation}
where $c^*$ is given by the formula \eqref{e.0315.nlp} for monotone environments, and $c_{LLW}$ is the minimal wave speed of \eqref{e.lv} corresponding to the invasion of the homogeneous state $(0,1)$ by the homogeneous state $(1,0)$ \cite{lewis2002spreading,li2005spreading}. It is interesting that in the expression \eqref{e.lv.nlp}, $c^*$ arises from local or nonlocal pulling, whereas $c_{LLW}$ represents the nonlinear effects of pushed waves. In particular, $c_2$ may
exceed the minimal wave speed of traveling fronts connecting $E_1$ and
$E_2$, and that its value depends on the first spreading speed $c_1$
in a non-increasing manner. The initial proof in \cite{girardin2019invasion} was based on delicate constructions of generalized super- and sub-solutions.
An alternative proof was later given in \cite{liu2019asymptotic,liu2021asymptotic}, using the Hamilton--Jacobi framework, for the weak competition case
$a,b\in (0,1)$ .

\medskip

\noindent {\bf Spreading dynamics of PDE systems without shifting habitat.} 
This simplified treatment via Hamilton-Jacobi equations enables us to later analyze the spreading of three
competing species all with compactly supported initial data \cite{liu2021stacked}. We are able to determine the (distinct) speed of each of the three competing species under certain situations.
To the best of our knowledge, this is the only work treating the spreading of three invasion fronts with distinct speeds in a system without comparison principle. 

Stacked invasion fronts have also been examined in other types of
systems. Examples include prey--predator models
\cite{ducrot2021asymptotic,ducrot2019spreading,LamLee2024} (See also \cite{wu2023propagation} for spreading result in periodic environment and \cite{ducrot2024spreading} for prey-predator system on a lattice),
SIS epidemic models \cite{chen2017longtime}; hunter-gatherer models \cite{xiao2021spreading},
and cooperative systems with identical diffusion rates \cite{iida2011stacked}.

\blue{Compared with traveling terraces \cite{Ducrot2014existence,polacik2020propagating} (a significant subject we have barely touched upon), in which typically all intermediate states are stable equilibria of the reaction term, stacked invasion fronts in systems seems to be able to sustain unstable states which are seemingly stable in the moving frame where they appear, but are convectively unstable (i.e. small perturbations grow while being advected away).}

With the exception of \cite{ducrot2021asymptotic,girardin2019invasion,liu2019asymptotic,liu2021stacked,liu2021asymptotic}, the spreading speeds identified in systems of PDEs in the literature are usually locally selected in the moving frame of the invasion front and are not
affected by the presence of other fronts propagating at different speeds.

Spreading phenomena in parabolic systems of a different type have also
been investigated in \cite{barles1990wavefront,freidlin1991coupled,JiangWu2025,li2025propagation}. In those works, the
systems are cooperative and irreducible when linearized about the
trivial equilibrium. Consequently, all components propagate with a
single common spreading speed, and stacked invasion fronts do not occur.

\medskip

\noindent {\bf Spreading dynamics of PDE systems involving shifting habitat.} 
We also mention some recent work devoted to the spreading dynamics in reaction-diffusion systems where the coefficients depends on the shifting heterogeneity $y=x-c_1 t$. Recent examples include   competitive systems with three species \cite{wang2024spreading}, as well as epidemic model of West Nile virus \cite{ahn2024spreading}.

\section{Spreading in Predator-Prey Systems}

\blue{For predator-prey systems a comparison principle is not immediately available and many studies regarding propagation phenomena in these systems have focused on the dynamics of traveling wave solutions. Until recently, few works have treated the spreading dynamics of predator-prey systems
with general {\it compactly supported} initial data. In \cite{pan1,pan2}, Pan determined the spreading speed of the predator for a
predator-prey system with initially constant prey density and compactly-supported predator.
Shortly after, Ducrot, Giletti, and Matano \cite{ducrot2019spreading} used methods from uniform persistence
theory to characterize the spreading dynamics when both predator and prey are initially
compactly supported. They showed that the behavior can be classified based on the speeds
of the prey in the absence of predator, and of the predator when prey is abundant. 

More recently, the Hamilton-Jacobi approach was applied to study the case when one or more of the coefficients depend on a shifting variable \cite{LamLee2024}. Such a situation 
appears biologically, e.g., with changes in the efficiency
of predator species in converting their prey to offspring \cite{Daugaard2019}; See also \cite{tao2026propagation} for nonlocal diffusion systems,  
 \cite{Chen2026propagation} for a stage-structured prey-predator system, and \cite{guo2023spreading,yang2025spreading} for systems with three or more species.

Below we discuss a recent breakthrough by Z. Jin \cite{Jin2026}, which establishes the equivalence between the spreading dynamics of a prey-predator system and an auxilliary Fisher-KPP equation. Consider
\begin{equation}
\left\{
\begin{aligned}
u_t &= d\, u_{xx} + u f(x-c_1t,u,v),
&& t>0,\; x\in\mathbb{R},\\
v_t  &= v_{xx} + v g(x-c_1t,u,v),
&& t>0,\; x\in\mathbb{R},
\end{aligned}
\right.    
\end{equation}
where \[
f:\mathbb{R}\times[0,\infty)\times[0,\infty)\to\mathbb{R}
\quad\text{and}\quad
g:\mathbb{R}\times[0,\infty)\times[0,\infty)\to\mathbb{R}
\] are locally Lipschitz continous and satisfy the following:
\begin{enumerate}

\item[(i)] \textbf{(Predation type)}
For each $u>0$ and $z\in\mathbb{R}$, the map
$v\mapsto f(z,u,v)$ is decreasing.
For each $v>0$ and $z\in\mathbb{R}$, the map
$u\mapsto g(z,u,v)$ is increasing.

\item[(ii)] \textbf{(Monostable)}
\[
f(\cdot,1,0)\equiv 0, \qquad \text{ and }\qquad 
\inf_{z\in\mathbb{R}} f(z,u,0)>0,
\qquad
\forall\,u\in[0,1).
\]

\item[(iii)] \textbf{(Strong KPP)}
For each $u\ge0$, the map
$v\mapsto g(z,u,v)$ is nonincreasing uniformly in $z\in\mathbb{R}$.

\item[(iv)] The limits
\( \displaystyle
\lim_{z\to\pm\infty} g(z,1,0) \)
exist and satisfies
\( \displaystyle
\inf_{z\in\mathbb{R}} g(z,1,0)>0.
\)
\end{enumerate}

For example, we can have
\[
f(z,u,v) = r(1-u-av) \quad \text{ and }\quad g(z,u,v) = k + bu - v
\]
with $r,a,k,b$ being positive function of $z$. 

Jin established the following pointwise estimate.
\begin{lemma}\label{jin:4.1}
Suppose $\inf_{\mathbb R} u_0>0$.
For each $\alpha>0$,
there exist constants $M_\alpha>0$ and $T_\alpha>0$ such that
\[
1-u(t,x)\le \alpha+M_\alpha v(t,x),
\qquad
\forall\, t\ge T_\alpha,\; \forall\, x\in\mathbb{R}.
\]
\end{lemma}
\begin{proof}
 See \cite[Lemma 4.1]{Jin2026}.   
\end{proof}
Hence, the population density $v$ of the predator satisfies the following:
\[
    v_t \leq  v_{xx} + v g(x-c_1 t, 1,v) \quad \text{ and }\quad v_t \geq v_{xx} + v G_\alpha(x-c_1 t, v)
\]
where $G_\alpha(z,v) = g(z,1-\alpha - M_\alpha v, v)$ is also nonincreasing in $v$. 

By repeating our previous arguments, one can deduce that the rate function $w = \lim_{\ep \to 0} -\ep \log v\left(\frac{t}{\ep}, \frac{x}{\ep} \right)$ satisfies
\[
\min\{w, w_t + |w_x|^2 + R(x/t)\} \geq 0 \geq \min\{w, w_t + |w_x|^2 + R_\alpha(x/t)\}
\]
with 
\[
R(x/t) = \limsup_{\ep \to 0 \atop (x',t') \to (x,t)} g(\frac{x' - c_1 t'}{\ep}, 1, 0) \]
and
\[R_\alpha(x/t) = \limsup_{\ep \to 0 \atop (x',t') \to (x,t)} G_\alpha(x'-c_1t',0) = \limsup_{\ep \to 0 \atop (x',t') \to (x,t)} g(\frac{x' - c_1 t'}{\ep}, 1-\alpha, 0). 
\]
Since $\alpha>0$ is arbitrary, we can let $\alpha \to 0$, so we again deduce the limiting Hamilton-Jacobi equation \eqref{hje.0722b} (with $R$ given above) and Theorem \ref{thm6.1} applies. 
In particular, it is proved that for a predator invading a region with established prey density (i.e. $\inf_{\mathbb{R}} u_0 >0$), the spreading speed of the predator coincides with the single species KPP problem by setting the density of the prey to be $1$ everywhere.}

\section{Entire Solutions and Front Interactions}

Entire solutions are solutions that are defined for all time $t\in\mathbb{R}$. 
They play a key role in the study of large-time dynamics of reaction--diffusion equations, since they often describe global-in-time patterns connecting different traveling fronts or equilibrium states.

In Section \ref{s.2}, we proposed to study the structure and asymptotic behavior of entire solutions of the scalar equation. Here we discuss some recent results concerning the entire solutions of \eqref{e.lv}. In particular, for equation \eqref{e.0312.1}, one may seek solutions that for $t\to -\infty$ behave like one or several traveling fronts and whose structure evolves as captured by the Cauchy problem as $t\to +\infty$.

\medskip

\noindent {\bf Entire solutions for the Lotka-Volterra competition system.} 
Only a limited number of results are currently available concerning entire solutions of \eqref{e.lv}. 
Morita and Tachibana \cite{morita2009entire} established the existence of entire solutions of merging-front type under the cases $a,b>1$ and $a<1<b$.  
More precisely, their construction produces solutions which, as $t\to -\infty$, resemble two traveling waves connecting the equilibria $E_1$ and $E_2$ which are moving toward each other from opposite directions.
As time increases, these fronts merge and the solution eventually converges uniformly in $x\in\mathbb{R}$ to the homogeneous stable state as $t\to +\infty$.

Further developments in the strong competition case were obtained in \cite{guo2019entire}, where 
{entire solutions formed by}
three or four merging bistable fronts are constructed. 
These results demonstrate that the interaction of several fronts can produce a rich class of entire solutions. {We also refer the reader to \cite{hao2022entire} for analogous results in the case of nonlocal dispersal.}

More recently, new types of entire solutions exhibiting different dynamical behaviors are constructed. 
These solutions originate from a traveling front of a single species
\[
\Phi_c(x-ct):=(\red{\phi_c(x-ct)},0)
\quad \text{as } t\to -\infty,
\]
while for $t\to +\infty$ it evolves into several distinct (monostable) fronts presenting as {\it diverging waves}, i.e.  they are moving away from each other. 
In particular, for the weak competition case $a,b<1$, it has been shown that the set of entire solutions (modulo trivial translations in space and time) forms a four-dimensional manifold, as parameterized by the wave speeds, with a limiting configuration depicted in Figure \ref{fig.2} (left panel). These entire solutions correspond to the case when we observe stacked invasion waves as $t\to\infty$ in the Cauchy problem. We note that the mechanism of {\it nonlocal pulling} imposes additional constraints in the intermediate fronts, so not all wave speeds can be prescribed. 

In a separate work, Salako obtained the existence of multiple families of {\it merging waves} \cite{salako2021invasion}. Interestingly, the conditions for the existence of such entire solutions is precisely those in which there is a slow {\it and} weak competitor. Under such circumstances, this species is always eliminated in the Cauchy problem (regardless of initial data), resulting in a trivial pattern of a traveling wave consisting of one single species if one looks forward in time. These classes of entire solution, however, demonstrate non-trivial patterns if one looks {\it backward in time}. 
Despite these advances, the general structure of entire solutions for \eqref{e.lv} remains largely open and presents many challenging questions.

\begin{figure}[htbp]
    \includegraphics[width=\textwidth]{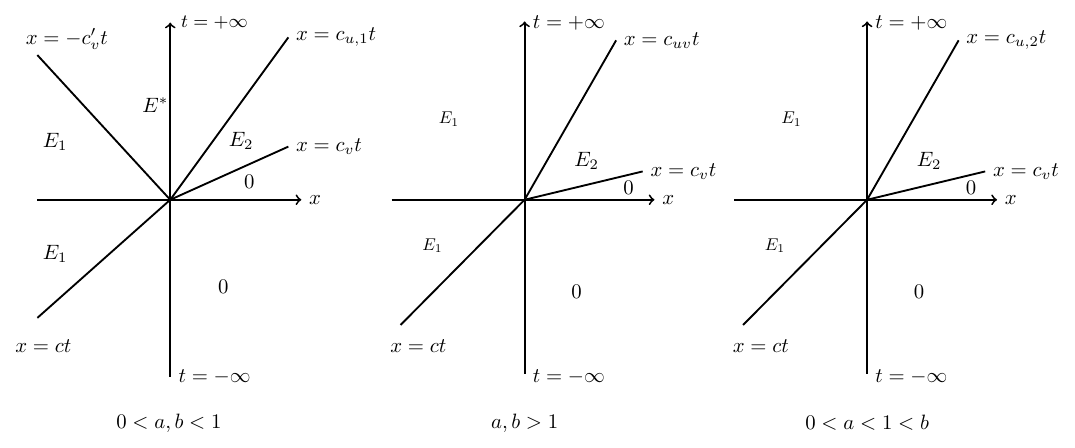}
    \caption{\small The hyperbolic limits of three families of entire solutions constructed in \cite{lam2020entire} in the form of diverging waves that connects $E_1$ as $x \to -\infty$ and ${\bf 0}$ as $x \to +\infty$. The second invasion front where $E_2$ is destabilized is nonlinearly selected. Here, ${\bf 0}=(0,0)$, $E_1 = (1,0)$, $E_2 = (0,1)$ and $E^*=(\tfrac{1-a}{1-ab},\tfrac{1-b}{1-ab})$.}
    \label{fig.2}
\end{figure}

\begin{problem}
    Characterize the set of all entire solutions to \eqref{e.lv}.
\end{problem}
While different classes of entire solutions, including traveling wave solutions, and diverging and merging waves are constructed \cite{morita2009entire,lam2020entire,salako2021invasion}, the characterization of the set of all traveling wave solutions remains a difficult question. For the state-of-the-art in this direction for the Fisher-KPP equation, see \cite{Hamel1999entire,Hamel2001travelling}.







\section{Bramson's Logarithmic Delay in Shifting Habitats}\label{s.11}


Here we discuss some recent result concerning the correction of order $\log t$ to the front position of solution for a class of reaction-diffusion equations with shifting environment. For the case of shifting habitat that is decreasing in the spatial variables, we prove some previously unpublished results; See Subsection \ref{ss:worse}.

Before doing so, we review  recent progress on the sharp spreading phenomena for the Fisher-KPP equation \cite{fisher1937wave,KolmogorovPetrovskyPiskunov1937}:
\begin{equation}\label{eq0}
    \begin{cases}
        u_t = u_{xx} + f(u),& \quad t>0,~x\in\mathbb{R},\\
        u(0,x) =u_0(x),& \quad x\in\mathbb{R},
    \end{cases}
\end{equation}
where $f$ has a monostable structure and satisfies the KPP condition:
\begin{align*}
  f(0)=f(1)=0,\quad  f>0\  \mbox{in $(0,1)$},\quad  f'(0)u\geq f(u)\    \mbox{for $u\in[0,1]$}.  
\end{align*}
The initial data $u_0$ satisfy 
$0\leq (\not\equiv)u_0\leq 1$ and $u_0(x)=0$ for $x\geq A$ with some $A\in\mathbb{R}$.
To locate the front, we set
\begin{align*}
\xi_\delta(t):=\sup\{x\in\R|\, u(t,x)\geq \delta\} \qquad \text{ for each }\delta\in(0,1).
\end{align*}
Then, from \cite{aronson1975nonlinear}, we see that 
\begin{align*}
 \xi_\delta(t)=c^*t+o(t),\quad t\to\infty,
\end{align*}
where $c^*:=2\sqrt{f'(0)}$. A natural question is what exactly $o(t)$ is?  Bramson \cite{bramson1983convergence} used probabilistic arguments to prove that  for any $\delta\in(0,1)$, 
    \begin{equation*}
         \xi_\delta(t)=c^*t-\frac{3}{2\lambda_*}\log t+O(1),\quad  t\to\infty,
    \end{equation*}
    where $\lambda_*=c^*/2$. Furthermore,
        there exists a uniformly bounded $\gamma(t)$ such that
\begin{align*}   \label{bramson1}
\lim_{t\to\infty}\sup_{x\in\mathbb{R}}\Big|u(t,x)-U_{c^*}\Big(x-c^*t+\frac{3}{2 \lambda_{*}} \log t+\gamma(t)\Big)\Big|=0,
\end{align*}
where $U_{c^*}$ is the decreasing traveling front with minimal speed $c^*$, normalized by $U_{c^*}(0)=1/2$, and 
$\gamma(t)$ converges to a constant  $\gamma_{\infty}$ depending on the initial data.
The logarithmic delay term is widely known as the Bramson correction term (or Bramson shift). 

Later, Hamel et al. \cite{hamel2013short} applied PDE arguments to recover this result. See also \cite{uchiyama1978behavior,lau1985nonlinear,nolen2017convergence,addario2019branching,bouin2020bramson}. A natural question is whether analogous logarithmic corrections appear in other dispersal models.
Besse et al. \cite{besse2023logarithmic} considered the following lattice model 
\begin{align*}
  \frac{d}{dt}u_j(t)=u_{j+1}+u_{j-1}-2u_{j}+f(u_j),\quad t>0,\ j\in\mathbb{Z}.  
\end{align*}
Graham \cite{graham2022bramson} and Roquejoffre \cite{roquejoffre2024dynamics} studied the 
nonlocal KPP equation:
\begin{align*}
u_t=J*u -u+f(u) &\quad t>0, \ x\in\R,
\end{align*} 
where $(J * u )(t,x):=\int_{\R}J(x-y)u(t,y)dy$ and $J$ is compactly supported.
In both lattice and nonlocal KPP models with step-like (sufficiently steep) initial data, the front location satisfies the same Bramson-type correction:
\[
\xi_\delta(t)=c^*t-\frac{3}{2\lambda_*}\log t+O(1),\qquad t\to\infty,
\]
where $c^*$ denotes the spreading speed corresponding to various models.

When $u_0$ is not compactly supported (on the right); however, the correction may change and can even become an advance; critical tails may generate additional 
$O(\log\log t)$ effects \cite{bramson1983convergence}.
Recently, Alfaro et al. \cite{alfaro2024bramson} revisited the problem entirely from a PDE perspective, in the flavor of \cite{hamel2013short}. Let us summarize the results as follows. First, without loss of generality, one may assume that $f'(0)=1$ so that $c^*=2$. Let $u_0\in[0,1]$ satisfy $\liminf_{x\to-\infty}u_0(x)>0$ and there exist $K_1,K_2>0$ and $\gamma\ge -2$ such that
\begin{align}\label{ic-AGX}
    K_1 x^{\gamma} e^{-x} \leq u_0(x) \leq K_2  x^{\gamma} e^{-x},\quad x>1.
\end{align}

\begin{theorem}[\cite{alfaro2024bramson,bramson1983convergence}]
Let $u(t,x)$ be the solution to \eqref{eq0}--\eqref{ic-AGX}. Then the following hold:
\begin{itemize}
    \item[(i)] If $\gamma>-2$, then there exists $C\geq 0$ such that
\[
\lim_{t\to +\infty}\ \inf_{|h|\le C}
\left\|\,u(t,\cdot)-U_{c^*}\Big(\cdot-2t+(\frac{1-\gamma}{2})\log t+h\Big)\right\|_{L^\infty(0,+\infty)}=0.
\]
  \item[(ii)]  If $\gamma=-2$, then there exists $C\geq 0$ such that
\[
\lim_{t\to +\infty}\ \inf_{|h|\le C}
\left\|\,u(t,\cdot)-U_{c^*}\Big(\cdot-2t+\frac{3}{2}\log t-\log \log t+h\Big)\right\|_{L^\infty(0,+\infty)}=0.
\]
\end{itemize}
\end{theorem}
As pointed out in \cite{alfaro2024bramson}, the sign of the logarithmic correction depends on the parameter $\gamma$. When $\gamma>1$, the front position exhibits a logarithmic advance.
When $-2<\gamma<1$, the front position exhibits a logarithmic delay. In the borderline case $\gamma=1$, the logarithmic correction vanishes, which means that the right tail of the initial data is comparable to that of the traveling front. For the critical case $\gamma=-2$, the correction is no longer purely logarithmic; an $O(\log \log t)$ term appears in the front location.

Beyond the effect of initial data in homogeneous media, recent research has focused on the propagation of fronts in shifting habitats. A natural question is whether the logarithmic delay persists when the front is "forced" by the shifting boundary of the habitat. This leads us to the following discussion.

\medskip

\noindent {\bf Log Delay in Shifting Habitats.}
Motivated by the understanding of the impact of climate change on the persistence and spreading of species 
\cite{berestycki2009can,potapov2004climate,li2014persistence}, we consider the problem \eqref{e.kpp.gen}, 
where the initial data satisfy 
\bea\label{ic-shift}
\begin{cases}
\red{0\leq u_0\leq 1}\\ 
\liminf\limits_{x\to-\infty}u_0(x)>0, \quad \text{ and }\quad  u_0(x)=0 \quad \text{ for all } x \gg 1,    
\end{cases}
\eea
 and the intrinsic growth rate $r(t,x)$ is piecewise constant in space and is determined by a shifting climate.
 More precisely, for some constants $R_1,R_2 >0$ and $\red{\beta},\eta\in \mathbb{R}$, we set
\bea\label{r-X}
r(t,x):=
\begin{cases}
R_1 &\text{ for }x \leq X(t),\\
R_2 &\text{ for }x > X(t),
\end{cases}
\quad X(t):=\beta t -\eta \log (t+1),\quad t\geq0.
\eea
Here $\beta$ represents the dominant climate shifting speed, while the logarithmic term introduces a correction, corresponding to a delayed ($\eta>0$) or advanced ($\eta<0$) interface relative to the constant-speed shift $\beta t$. In fact, $\eta\neq0$ is motivated by the study in co-invasion dynamics \cite{holzer2014accelerated,girardin2019invasion,liu2021stacked}, where the faster invader may generate a logarithmic shift {\cite{peng2021sharp,wu2023sharp}.

Note that $r(t,x)$ has a jump discontinuity across $x=X(t)$. Nevertheless, 
\eqref{e.kpp.gen} 
is well-posed in the weak sense, while the solution becomes classical and enjoys standard parabolic regularity away from the interface.

\medskip

%
\subsection{Shifting Environments with Favorable Region Ahead}
%
In \cite{LamWu_preprint}, the authors consider the problem 
\eqref{e.kpp.gen} 
with
\begin{align}\label{r-1}
 r(t,x):=
\begin{cases}
  1-a &\text{ for }x \leq X(t),  \\
    1 &\text{ for }x > X(t),
\end{cases}
\quad \mbox{ for some $a\in(0,1)$};   
\end{align}
namely, $R_1=1-a$ and $R_2=1$ in \eqref{r-X}.
In this setting, the region ahead of the invasion front is more favorable, since the intrinsic growth rate is higher there. Consequently, the population is chasing a more suitable habitat ahead of its leading edge, and the interface $x=X(t)$ denotes the left boundary of this favorable region.

It has been shown \cite{holzer2014accelerated,Lam2022asymptotic}  that the spreading speed $c_*$ of 
\eqref{e.kpp.gen} 
exists and is expressed by
\bea\label{c*}
{c}_* =\begin{cases}
2 &\text{ if } \beta \le 2,\\
\dps 
\lambda_*+\frac{1-a}{\lambda_*}
&\text{ if } 2 < \beta < 2(\sqrt a +\sqrt{1-a}),\\
2\sqrt{1-a}=:c_{\min}  &\text{ if }  \beta \geq 2(\sqrt a +\sqrt{1-a}),
\end{cases}
\eea
where the exponent $\lambda_*\in (1-\sqrt{a}, \sqrt{1-a})$ is independent of $\eta$ and is given by
\begin{align}\label{lambda*}
 \lambda_*=\frac{\beta}{2}-\sqrt{a}.
\end{align}

The first two authors of the present paper determined the front location for \eqref{e.kpp.gen}, where $r(t,x)$ is defined in \eqref{r-1}. 
\red{Recall that $X(t):=\beta t -\eta \log (t+1)$.
We consider the following four cases:}
\begin{itemize}
\red{
\item[{\bf(A0)}] (Spreading ahead of environment) 
$(\beta,\eta) \in  (-\infty,2)\times \mathbb{R}$ or $(\beta,\eta) \in \{2\}\times (\tfrac32, \infty)$. 

\item[{\bf(A1)}] (Supercritical pulling) 
   $(\beta,\eta) \in  (2,2(\sqrt{a} + \sqrt{1-a}))\times \mathbb{R}$ or $(\beta,\eta) \in \{2\}\times (-\infty, 1/2)$.
   }
    \item[{\bf(A2)}]  (Critical pulling) $\beta =  2(\sqrt{a} + \sqrt{1-a})$ \red{and $\eta\in\R$}. 
    \item[{\bf(A3)}] (No pulling) $\beta> 2(\sqrt{a}+\sqrt{1-a})$ \red{and $\eta\in\R$}.
\end{itemize}

\red{Let $R\in\{1-a,1\}$.}
Denote by \red{$\Phi_{\lambda,R}$} the traveling wave solution that
satisfies
\begin{equation}\label{TW-B}
\begin{cases}
c_\lambda \Phi' + \Phi'' +\Phi (\red{R} -\Phi) =0\quad \text{ for }z \in \mathbb{R},\\
\Phi(-\infty) =\red{R}, \qquad  \Phi(+\infty) = 0,
\end{cases}
\end{equation}
with a suitable spatial translation such that
\bea\label{TW-AS}
\begin{cases}
\lim\limits_{z \to +\infty} e^{\lambda z}\Phi(z) = 1,& \quad \text{ if } c_{\lambda}>2\sqrt{\red{R}}\quad (\lambda\in(0,\sqrt{\red{R}}\,)),\\
\lim\limits_{z \to +\infty} z^{-1}e^{\lambda z}\Phi(z) = 1,& \quad \text{ if } c_{\lambda}=2\sqrt{\red{R}}
\quad (\lambda=\sqrt{\red{R}}\,),
\end{cases}
\eea
where $c_\lambda:=\lambda+\frac{\red{R}}{\lambda}$, and write
$\Phi_{\min,\red{R}}:=\Phi_{\red{\sqrt{R}},\red{R}}$.

We have the following result.

\begin{theorem}[\cite{LamWu_preprint}]\label{thm:lamwu}
Let $u(t,x)$ be the solution to 
\eqref{e.kpp.gen} with initial data satisfying \eqref{ic-shift}. Then the following hold:
\begin{itemize}
  \red{ \item[(i)] Assume that {\bf(A0)} holds. Then there exists a bounded function $\Lambda(t)$ such that for any $\delta>0$, 
\begin{align*}
\sup_{x \geq X(t) + \delta \log t }\Big|u(t,x) - \Phi_{\min,1} \Big(x-2 t +\frac{3}{2}  \log t+\Lambda(t)\Big)\Big| \to 0 \text{ as } t\to\infty.    
\end{align*}
}
    \item[(ii)] Assume that {\bf(A1)} holds. 
Then there exists a bounded function $\Lambda(t)$ such that 
\begin{align*}
    \sup_{x\in \R }\Big|u(t,x) - \Phi_{\lambda_*,\red{1-a}}\Big(x-c_* t +\frac{1}{\lambda_*}\left(\frac{3}{2} - \sqrt{a}\eta\right) \log t+\Lambda(t)\Big)\Big| \to 0 \text{ as } t\to\infty,
\end{align*}
where $\lambda_*$ is defined in \eqref{lambda*}.
 \item[(iii)] Assume that {\bf(A2)} holds. 
 Set $q=-\tfrac32 + \eta \sqrt a$. Then there exists a bounded function $\Lambda(t)$ such that 
\begin{align*}
\sup_{x\in \R }\Big|u(t,x) - \Phi_{\min,\red{1-a}}\Big(x-{m}_q(t)+\Lambda(t)\Big)\Big| \to 0 \text{ as } t\to\infty,
\end{align*}
where \begin{equation}\label{mq}
{m}_q(t) := 
\begin{cases}
c_{\min} t - \frac{3}{2\lambda_{\min}} \log t , &\text{ if }q< -2,\\
c_{\min} t - \frac{3}{2\lambda_{\min}} \log t + \frac{1}{\lambda_{\min}} \log \log t &\text{ if }q= -2,\\
c_{\min} t  + \frac{q-1}{2\lambda_{\min}} \log t &\text{ if }q > -2.
\end{cases}
\end{equation}
\red{Here $c_{\min}=2\sqrt{1-a}$ and $\lambda_{\min}=\sqrt{1-a}$.}
\item[(iv)] Assume that {\bf(A3)} holds. 
Then there exists a bounded function $\Lambda(t)$ such that 
\begin{align*}
\sup_{x\in\R}  \Big|u(t,x )  -\Phi_{\min,\red{1-a}}\Big(x- c_{\min} t+ \frac{3}{2\lambda_{\min}}\log t + \Lambda(t)\Big) \Big|  \to 0\quad \text{ as }t\to\infty,
\end{align*}
\red{where $c_{\min}=2\sqrt{1-a}$ and $\lambda_{\min}=\sqrt{1-a}$.}
\end{itemize}
\end{theorem}

\red{

The profile convergence in the regime
$\beta=2$, $\frac12\le\eta\leq \frac32$
remains open. We conjecture that the conclusion (ii) of Theorem~\ref{thm:lamwu} continues to hold when $\beta=2$ and $\frac12\le\eta<\frac32$.
The case $(\beta,\eta)=(2,\frac{3}{2})$ is marginal since the interface $X(t)$ and the classical KPP front have the same logarithmic location. In this case, the level-set location is established (\cite[Theorem 1.3]{LamWu_preprint}): 
\begin{align*}
 \xi_{\delta}(t)={2 t -\frac{3}{2}\log t+ O(1)}.    
\end{align*}
However, the corresponding profile convergence problem remains open.

}



These results establish refined asymptotic descriptions of the invasion front and connect Hamilton--Jacobi theory with finer front dynamics.
In addition, when $\eta = 0$, i.e. $X(t)  =\beta t$ for some $\beta>0$, 
Theorem~\ref{thm:lamwu} \red{gives}
$$
\xi_{\delta}(t) = \begin{cases}
{2 t -\frac{3}{2}\log t+ O(1)} & {\text{ if } \beta<2,} \\
    c_* t - \frac{3}{2\lambda_*}\log t + O(1)  &\text{ if }2 \leq  \beta <2(\sqrt a + \sqrt{1-a}),\\
    c_{\min}t - {\frac{5}{4 \lambda_{\min}} \log t}+ O(1)&\text{ if }\beta = 2(\sqrt a + \sqrt{1-a}),\\
    c_{\min}t - \frac{3}{2 \lambda_{\min}} \log t+ O(1)  &\text{ if }\beta > 2(\sqrt a + \sqrt{1-a}),
\end{cases}
$$
where 
$$
\lambda_* = \frac{\beta}{2}-\sqrt{a}, \quad c_* =  \lambda_* + \frac{1-a}{\lambda_*},\quad \lambda_{\min}=\sqrt{1-a},\quad c_{\min} = 2\sqrt{1-a}.
$$
We remark that a new Bramson correction term 
\[-\frac{5}{4\lambda_{\min}} \log t\]
appears in the transition case $\beta= 2(\sqrt a + \sqrt{1-a})$.

\medskip

\subsection{Shifting Environments with Worse Region Ahead}\label{ss:worse} 
%
%
%
We consider the problem 
\eqref{e.kpp.gen} 
with
\beaa
r(t,x):=
\begin{cases}
  1 &\text{ for }x \leq X(t),  \\
    1-a &\text{ for }x > X(t),
\end{cases}
\quad \mbox{ for some $a\in(0,1)$};
\eeaa
namely, $R_1=1$ and $R_2=1-a$ in \eqref{r-X}.
This means that the region ahead of the invasion front is less favorable.

Previous work shows that (\cite{berestycki2018forced,holzer2014accelerated,Lam2022asymptotic}) that the spreading speed $c_*$ of 
\eqref{e.kpp.gen} exists and is expressed by
\bea\label{c*2}
{c}_* =\begin{cases}
2\sqrt{1-a} &\text{ if } \beta \le 2\sqrt{1-a},\\
\beta &\text{ if } 2\sqrt{1-a} < \beta  \le 2,\\
2  &\text{ if }  \beta> 2.
\end{cases}
\eea

We next derive asymptotic bounds for the front location of \eqref{e.kpp.gen} .

\begin{theorem}\label{thm:less-favorable}
Let \red{$\delta\in(0,1-a)$} be given. Then there exists $C\geq0$ and $T>0$ such that the following hold:
\begin{itemize}
\item[{\rm(i)}] If $ \beta<2\sqrt{1-a}$ and $\eta\in\R$, then  for $t\geq T$,
\[
2\sqrt{1-a}t-\frac{3}{2\lambda_{\min}}\log t-C \leq \xi_\delta (t)\leq 2\sqrt{1-a}t-\frac{3}{2\lambda_{\min}}\log t+C 
\]  
where $\lambda_{\min}:=\sqrt{1-a}$.
\item[{\rm(ii)}]  If $2\sqrt{1-a} < \beta<2$ and $\eta\in \R$, then
\begin{equation*}
 \xi_\delta(t)\le \beta t-\eta\log(t+1)+C,
\quad  \mbox{for $t\geq T$}.
\end{equation*}
\item[(iii)] If $0<\beta<2$ and $\eta\geq 0$, then  
\begin{equation*}
 \xi_\delta(t)\ge \beta t-\eta\log(t+1)-C,
\quad  \mbox{for $t\geq T$}.
\end{equation*}
\end{itemize}
\end{theorem}

\begin{corollary}
Let \red{$\delta\in(0,1-a)$} be given. 
If $2\sqrt{1-a}\leq  \beta <2$ and $\eta=0$, then there exists $C\geq0$ and $T>0$ such that
\begin{equation}\label{cor-ineq}
\beta t-C \le \xi_\delta(t)\le \beta t+C,
\quad  \mbox{for $t\geq T$}. 
\end{equation}
\end{corollary}
\begin{proof}
\red{The argument of (ii) in Theorem~\ref{thm:less-favorable} can be suitable modified to cover the case $\beta=2\sqrt{1-a}$ and $\eta=0$.}
Combining (iii)  in Theorem~\ref{thm:less-favorable}, we obtain \eqref{cor-ineq}.
\end{proof}
\begin{remark}
\blue{In \cite[Theorem 1.5(iii)]{berestycki2018forced}, the convergence to the minimal forced traveling wave is established.  }
\end{remark}

\begin{problem}
The behavior of $\xi_\delta(t)$ remains open in the critical case $\beta=2\sqrt{1-a}$ with $\eta\neq 0$, as well as in the case $\beta\geq 2$. In case $\beta = 2$, whether the level set remains in bounded distance with $X(t)$ depends on the value of $\eta$.
\end{problem}

\medskip

\red{
To prove Theorem~\ref{thm:less-favorable}, we need several lemmas.
We first treat assertion (i) of Theorem~\ref{thm:less-favorable}.
When $\beta<2\sqrt{1-a}$, the moving interface $x=\beta t-\eta\log(t+1)$ moves strictly slower than the minimal KPP speed  $2\sqrt{1-a}$ in the less favorable region. So the large time level set is governed by the homogeneous Fisher--KPP dynamics with growth rate $1-a$. 
Hence, part (i) of Theorem~\ref{thm:less-favorable} follows from a comparison argument combined with the
classical level-set asymptotics for the homogeneous Fisher--KPP equation \cite{bramson1983convergence,hamel2013short}.
}

\red{
\begin{lemma}\label{lem:small-beta}
Let $\beta<2\sqrt{1-a}$ and $\delta\in(0,1-a)$ be given. Then there exists $T>0$ and $C>0$ such that 
\bea\label{beta-small}
2\lambda_{\min} t-\frac{3}{2\lambda_{\min}}\log t-C \leq \xi_\delta (t)\leq 2\lambda_{\min}t-\frac{3}{2\lambda_{\min}}\log t+C\quad  \mbox{for $t\geq T$},
\eea
where $\lambda_{\min}:=\sqrt{1-a}$.
\end{lemma}
\begin{proof}
For the lower bound, let $w$ solve
\[
w_t=w_{xx}+w(1-a-w),\quad t>0,\ x\in\R,\quad w(0,x)=u_0(x),\ x\in\R.
\]
Since $r(t,x)\geq 1-a$, by comparison, 
$u\geq w$ and thus
$\xi_\delta(t)\geq\widetilde\xi_\delta(t)$,
where $\widetilde\xi_\delta(t)$ denotes the $\delta$-level set of $w$. By the classical
Fisher--KPP level-set asymptotics
\cite{bramson1983convergence,hamel2013short},
\[
\widetilde\xi_\delta(t)=2\lambda_{\min}t
-\frac{3}{2\lambda_{\min}}\log t
+O(1),
\]
and the desired lower bound follows.

For the upper bound, choose $\bar c\in (\beta, 2\lambda_{\min})$ and fix $L>0$ such that $\bar c t+L>X(t)$ for all $t\geq 0$. Hence, $r(t,x)=1-a$ for all $x\geq \bar c t+L$. Let $v$ be the solution of
\[
\begin{cases}
v_t=v_{xx}+v(1-a-v),
& t>0,\ x\in\mathbb R,\\
v(0,x)=\mathbf{1}_{(-\infty,A]}(x),
\end{cases}
\]
where $A>L$ is chosen so that
$u_0(x)\le \mathbf{1}_{(-\infty,A]}(x)$ for all $x\in\R$.

Since $\bar c<2\lambda_{\min}$, the spreading result for the homogeneous Fisher--KPP equation yields
$v(t,\bar c t+L)\to 1-a$ as
$t\to\infty$. Also, note that $v(0,L)=1>0$.
Hence, there exists $\delta_0\in(0,1)$ such that $v(t,\bar c t+L)\geq\delta_0$ for all $t\ge0$.

Define
$\bar u:=\frac{1}{\delta_0}v$.
A direct computation gives
\[
\bar u_t-\bar u_{xx}-\bar u(1-a-\bar u)
=
\frac{1-\delta_0}{\delta_0^2}v^2
\geq 0 \ \text{ in } \Omega:=\{(t,x):x\ge\bar c t+L\}.
\]
Furthermore, we have 
$\bar u(0,x)\ge u_0(x)$,
and
\[
\bar u(t,\bar c t+L)
= \frac{v(t,\bar c t+L)}{\delta_0}
\geq 1 \geq  u(t,\bar c t+L),\quad t\geq 0.
\]
Therefore, by the comparison principle, $u(t,x)\le\bar u(t,x)$
for $(t,x)\in\Omega$.

Let $\widetilde\xi_{\kappa}(t)$
denote the $\kappa$-level set of $v$. If $x\ge\bar c t+L$ and 
$u(t,x)\geq\delta$, 
then  $v(t,x)\geq \kappa:=\delta_0\delta$. Hence,
$\xi_\delta(t)\leq \max\{\bar c t+L, \widetilde\xi_{\delta_0\delta}(t)\}$.
Also, from \cite{bramson1983convergence,hamel2013short}, we see that
\[
\widetilde\xi_{\delta_0\delta}(t)
=
2\lambda_{\min}t
-\frac{3}{2\lambda_{\min}}\log t
+O(1).
\]
Using $\bar c<2\lambda_{\min}$, we 
have $\widetilde\xi_{\delta_0\delta}(t)>\bar c t+L$ for all large $t$, which implies that
$\xi_\delta(t)\leq \widetilde\xi_{\delta_0\delta}(t)$ for all large $t$. 
We therefore prove the desired upper bound. This completes the proof.
\end{proof}
}

We 
\red{next} treat assertion (ii) of Theorem~\ref{thm:less-favorable}.

\begin{lemma}\label{lem:middle-beta}
Let $2\sqrt{1-a}<\beta<2$, $\eta\in\R$,
and $\delta\in(0,1-a)$ be given. 
Then there exist $T>0$ and $C>0$ such that
\begin{equation*}
\xi_\delta(t)\le \beta t-\eta\log(t+1)+C,
\qquad t\ge T.
\end{equation*}
\end{lemma}
\begin{proof}
First, we set
\bea\label{y-v}
y:=x-\beta t+\eta\log(t+1),\quad 
v(t,y):=u\big(t, y+\beta t-\eta\log(t+1)\big).
\eea
A direct computation yields
\begin{equation}\label{v-eq}
v_t=v_{yy}+\Big(\beta-\frac{\eta}{t+1}\Big)v_y+v(g(y)-v),\quad t>0,\ y\in\R,
\end{equation}
where $g(y)=1$ if $y\leq 0$; $g(y)=1-a$ if $y>0$.

Now, define
\bea\label{zeta-set}
\zeta_\delta(t):=\sup\{y\in \R:\ v(t,y)>\delta\}.
\eea
Then 
\begin{equation}\label{xi-zeta}
\xi_\delta(t)=\beta t-\eta\log(t+1)+\zeta_\delta(t).
\end{equation}
It suffices to prove that there exist $T>0$ and $C>0$ such that
\begin{equation}\label{zeta-upper}
 \zeta_\delta(t)\le C,\quad \text{ for } t\ge T.
\end{equation}

To construct a generalized supersolution via a forced wave, we introduce a smooth non-increasing function 
\red{$\tilde g\in C^{\infty}(\mathbb{R})$ satisfying}
\beaa
\tilde g(y)=\begin{cases}
1,& \quad \text{ if } y\leq 0,\\
1-a,& \quad \text{ if } y\geq 1.
\end{cases}
\eeaa
Therefore, $\tilde g(y)\geq g(y)$ for all $y\in \R$.
By \cite[Theorem 1.3 (iii)]{berestycki2018forced},
we can take a (non-minimal) profile $U_{\beta}$ solving
\begin{equation}\label{eq:Ub}
\begin{cases}
\beta U'+U''+(\tilde g(y)-U)U=0,\quad y\in\R,\\
U(-\infty)=1,\quad U(+\infty)=0,\\
 U'(y)<0,\quad y\in\R,\\
U(y)\sim e^{-\lambda_0 y},\quad y\to+\infty,
\end{cases}
\end{equation}
where $\lambda_0=\frac{\beta-\sqrt{\beta^2-4(1-a)}}{2}>0$. Note that such a non-minimal front has a slower decay rate.

Fix $M_1\ge2$. We define the following stationary gluing function
\beaa
\bar v(y):=
\begin{cases}
M_1U_\beta(y),& \quad  y< 1,\\
\min\{ M_1U_\beta(y), \psi(y)\},&  1 \leq y\leq Y_1,\\
\psi(y),&\quad y> Y_1,\\
\end{cases}
\eeaa
where $Y_1>1$ will be chosen below and $\psi$ is defined by
\beaa
\psi(y):=M_2 e^{-\mu y }.
\eeaa
Here  $\mu>0$ is chosen such that $\mu\in\Big(\lambda_0,\frac{\beta+\sqrt{\beta^2-4(1-a)}}{2}\Big)$ and thus
\bea\label{nu-1}
\nu:=\beta \mu -\mu^2 -(1-a)>0.
\eea
Moreover, 
$M_2>0$ is chosen such that
\bea\label{M2}
M_2>M_1 U_{\beta}(1) e^{\mu}.
\eea

Define 
\beaa
\mathcal L[\phi]
&:= \phi_{t}-\phi_{yy}-\Big(\beta-\frac{\eta}{t+1}\Big) \phi_y-\phi(\tilde g-\phi).
\eeaa
Using \eqref{eq:Ub}, we obtain that for all $t>0$ and $y\leq Y_1$,
\begin{align*}
\mathcal L[M_1U_\beta]
&:= (M_1U_\beta)_{t}-(M_1U_\beta)_{yy}-\Big(\beta-\frac{\eta}{t+1}\Big)(M_1U_\beta)_y-M_1U_\beta(\tilde g-M_1U_\beta)\notag\\
&= M_1\Big[-U_{\beta}''-\beta U_{\beta}'-U_{\beta}(\tilde g-U_{\beta})\Big]
+\frac{\eta}{t+1}M_1 U_{\beta}'+M_1(M_1-1)U_{\beta}^2\notag\\
&= \frac{\eta}{t+1}M_1 U_{\beta}'+M_1(M_1-1)U_{\beta}^2.
\end{align*}

Set 
\begin{align*}
K_0:=\|U'_{\beta}\|_{L^{\infty}},\quad \ep_0:=\inf_{y\leq Y_1}U_{\beta}(y)=U_{\beta}(Y_1)>0.
\end{align*}
Then, we can find $t_1\gg1 $ such that
\begin{equation}\label{t_1}
\mathcal L[M_1U_\beta]\ge  -\frac{|\eta|}{t+1}M_1K_0+M_1(M_1-1)\ep_0^2  \ge0\qquad \text{for all }t\ge t_1,\ y\in(-\infty,Y_1].
\end{equation}

On the other hand, by \eqref{nu-1}
and the fact that $\tilde g(y)=1-a$ for $y\ge 1$, we see that, for all $t>0$ and $y\ge 1$,
\begin{align*}
\mathcal L[\psi]
= M_2e^{-\mu y}\Big(\nu -\frac{\eta }{t+1}\mu +M_2{e^{-\mu y}}\Big).
\end{align*}
Hence, there exists $t_2\gg1$ such that 
\begin{equation}\label{t_2}
\mathcal L[\psi]\ge0\qquad \text{for all }t\ge t_2,\ y\in[1,\infty).
\end{equation}

Furthermore, thanks to \eqref{M2}, we have
\beaa
M_1U_{\beta}(1)<M_2 e^{-\mu }=\psi(1).
\eeaa
Also, since $\mu>\lambda_0$, we may choose $Y_1>1$ large enough such that $\psi(Y_1)<M_1U_\beta(Y_1)$.
\red{Therefore, $\bar v$ is continuous at both $y=1$ and $y=Y_1$.}
Combining \eqref{t_1}--\eqref{t_2} and $\tilde{g}\ge g$, 
we see that $\bar v$ is a generalized supersolution of \eqref{v-eq} for  $t\geq T:=\max\{t_1,t_2\}$ and $y\in \R$ (cf. \cite[Remark 1.1.2]{lam2022introduction}).

To apply the comparison principle from $t=T$, we show that there exists $K>1$ such that
\bea\label{ic-T}
v(T,y)\le K\bar v(y)\quad \text{for all }y\in\R.
\eea
Indeed, since $\bar v(y)>0$ on $\R$, $\bar v(y)\ge \kappa>0$ on $(-\infty,Y_1]$ for some $\kappa>0$, and $v(T,\cdot)$ is bounded, 
$v(T,y)/\bar v(y)$ is bounded on $(-\infty,Y_1]$. On the other hand, note that for $y\ge Y_1$,
$\bar v(y)=M_2e^{-\mu y}$. Since $u_0$ is compactly supported on the right, heat kernel estimates yield that
\[
v(T,y)=o(e^{-\mu y})\quad \text{as }y\to+\infty.
\]
So $v(T,y)/\bar v(y)$ is also bounded on $[Y_1,\infty)$, and thus there exists $K>1$ such that \eqref{ic-T} holds.
Moreover, it is not hard to see that $K\bar{v}$ is still a generalized supersolution for $t\geq T$ and $y\in \R$ using $K>1$. 
\red{
Indeed, there exists a neighborhood of $y=1$ on which $K\bar v=KM_1U_\beta$, and a neighborhood of $y=Y_1$ on which $K\bar v=K\psi$. Hence,
the gluing pattern of $\bar v$ near $y=1$ and $y=Y_1$ does not change after multiplication by $K$. Moreover, 
\[
\mathcal L [K\bar v]
=K \mathcal L[\bar v]+K(K-1)\bar v^2\ge0.
\]
Therefore, $K\bar v$ is also a generalized supersolution of \eqref{v-eq} for $t\ge T$.

}

Consequently, we may use the comparison principle to assert that
\beaa
v(t,y)\le \red{K}{\bar v}(y)\quad \text{ for all } t\ge T,\ y\in\R.
\eeaa
In particular, \red{since $K\bar v(y)\to0$ as $y\to+\infty$,}
for fixed $\delta\in(0,1-a)$, we can choose $C_+\gg1$ so that $v(t,y)\leq   \red{K}{\bar v}(y)\leq \delta$ for all $t\geq \red{T}$ and $y\geq C_+$, which implies \eqref{zeta-upper}. This completes the proof.
\end{proof}

We next derive the lower bound for $\xi_{\delta}(t)$.

\begin{lemma}\label{lem:middle-beta2}
Let $\beta<2$, $\eta\geq 0$ and \red{$\delta\in(0,1-a)$} be given. 
Then there exist $T>0$ and $C>0$ such that
\begin{equation*}
\xi_\delta(t)\ge \beta t-\eta\log(t+1)-C,
\qquad t\ge T.
\end{equation*}
\end{lemma}
\begin{proof}
Let us set $u(t,x)=v(t,y)$ as in \eqref{y-v}. Then $v(t,y)$ satisfies \eqref{v-eq}. Denote $\zeta_{\delta}(t)$ as in \eqref{zeta-set}. As in the proof of Lemma~\ref{lem:middle-beta}, it suffices to prove that 
there exist $T>0$ and $C>0$ such that
\begin{equation}\label{zeta-lowerbd}
 \zeta_\delta(t)\ge  -C,\quad \text{ for } t\ge T.
\end{equation}

Define
\[
\phi(y):=e^{-\frac{\beta}{2}y}\sin\Big(\frac{y}{\ell}\Big),\quad y\in(0,\pi\ell),
\]
where $\ell>0$ to be chosen later. Let $y_*\in(0,\pi \ell)$ be the maximum point such that 
$\phi(y_*)=\max_{y\in[0,\pi\ell]}\phi(y)$.
Choose $L_0>\pi\ell$ and define, for $y\in\R$,
\begin{equation}\label{lower-v}
\underline v(y):=
\begin{cases}
\kappa \phi(y_*),&  y+L_0\leq y_*,\\
\kappa\phi(y+L_0),& y_*<y+L_0<\pi\ell,\\[2mm]
0,& \text{otherwise},
\end{cases}
\end{equation}
where $\kappa\in(0,1)$ will be fixed later.
By the choice of $L_0$, if $\underline v(y)>0$, then $y\in(-\infty,-L_0+\pi\ell)\subset(-\infty,0)$.
Therefore,
\begin{equation}\label{g=1-supp}
g(y)=1\qquad\text{ if } \underline v(y)>0.
\end{equation}
By \eqref{g=1-supp}, on the set $\{\underline v>0\}$ the equation for $v$
reduces to the Fisher--KPP part with growth rate $1$.
Let
\[
\mathcal{L}[w]
:=w_t-w_{yy}-\Big(\beta-\frac{\eta}{t+1}\Big)w_y-w(1-w),
\]
so that $v$ satisfies $\mathcal{L}[v]=0$ whenever $g(y)=1$.

A direct computation shows that $\phi$ solves
\begin{equation}\label{phi-ode}
\phi''(y)+\beta \phi'(y)+\Big(\frac{\beta^2}{4}+\frac{1}{\ell^2}\Big)\phi(y)=0,
\qquad y\in(0,\pi\ell).
\end{equation}
Since $0<\beta<2$, we may choose $\ell>0$ sufficiently large that
\begin{equation*}
\sigma:=1-\frac{\beta^2}{4}-\frac{1}{\ell^2}>0.
\end{equation*}

On the interval $\{y_*<y+L_0<\pi\ell\}$, we have $\underline v(y)=\kappa\phi(y+L_0)$.
Then we have
\begin{align*}
\mathcal{L}[\underline v]
&= -\kappa\phi''(s)-\Big(\beta-\frac{\eta}{t+1}\Big)\kappa\phi'(s)-\kappa\phi(s)\big(1-\kappa\phi(s)\big),
\qquad s:=y+L_0.
\end{align*}
Using \eqref{phi-ode}, a direct computation yields 
\begin{align}\label{eq:L-phi}
\mathcal{L}[\underline v]
&=\kappa\Big(\frac{\beta^2}{4}+\frac{1}{\ell^2}\Big)\phi(s)
+\kappa\frac{\eta}{t+1}\phi'(s)
-\kappa\phi(s)\big(1-\kappa\phi(s)\big) \notag\\
&=\kappa\phi(s)\Big(\frac{\beta^2}{4}+\frac{1}{\ell^2}-1+\kappa\phi(s)\Big)
+\kappa\frac{\eta}{t+1}\phi'(s) \notag\\
&= -\kappa\phi(s)\big(\sigma-\kappa\phi(s)\big)+\kappa\frac{\eta}{t+1}\phi'(s).
\end{align}
Recall that $\phi$ is increasing on $(0,y_*)$ and decreasing on $(y_*,\pi\ell)$.
Therefore, the assumption $\eta\ge 0$ yields
\[
\kappa\frac{\eta}{t+1}\phi'(s)\le 0,\qquad t\ge 0,\ s\in(y_*,\pi\ell).
\]
Choose $\kappa>0$ sufficiently small such that $0\leq \kappa\phi(s)\le \sigma$ for all $s\in[y_*,\pi\ell]$.
From \eqref{eq:L-phi} we see that
\beaa
\mathcal{L}[\underline v]\le -\kappa\phi(s)\big(\sigma-\kappa\phi(s)\big)\leq 0
\quad \text{for } t\ge 0,\  y_*<y+L_0<\pi\ell.
\eeaa

On $\{y+L_0\le y_*\}$, $\underline v\equiv \kappa\phi(y_*)$ is constant. Then 
\[
\mathcal{L}[\underline v]=-\kappa\phi(y_*)\big(1-\kappa\phi(y_*)\big)\le 0,
\]
since we may choose $\kappa$ smaller such that $\kappa\phi(y_*)\in(0,1)$.

Note that $\phi'(y_*)=0$, so $\underline v$ is $C^1$ there. At the right endpoint $y+L_0=\pi\ell$, we have $\underline v=0$ and the left derivative
$\partial_y\underline v<0$; while the right derivative is $0$. Therefore,  $\underline v$
is a generalized subsolution.

Moreover, since $\liminf_{x\to-\infty}u_0(x)>0$,
if necessary we can take $\kappa>0$ smaller and $L_0$ larger so that
\begin{equation}\label{eq:init-compare}
\underline v(y)\le v(0,y)\quad\text{for all }y\in\R.
\end{equation}
Therefore, we can apply the comparison principle to conclude that 
\begin{equation}\label{eq:compare}
v(t,y)\ge \underline v(y)\quad\text{for all }t\ge 0,\ y\in\R.
\end{equation}
In particular, for any $\delta\in(0, \kappa \phi(y_*))$, there exists $C>0$
\[
\zeta_\delta(t)\ge -C
\qquad\text{for all }t\ge 0.
\]
Therefore \eqref{zeta-lowerbd} holds for any small $\delta>0$. In fact, \eqref{zeta-lowerbd} still holds for any \red{$\delta\in(0,1-a)$} (with $T=T(\delta)$) due to the fact that 
$\zeta_{\delta_1}(t)-\zeta_{\delta_2}(t)=O(1)$ for any $\delta_1,\delta_2\in(0,1-a)$.
This completes the proof.
\end{proof}

We are ready to prove Theorem~\ref{thm:less-favorable}.

\begin{proof}[Proof of Theorem~\ref{thm:less-favorable}]
\red{Part (i) follows from Lemma~\ref{lem:small-beta}.}
Parts (ii) and  (iii)
of Theorem~\ref{thm:less-favorable} \blue{follow from}
Lemma~\ref{lem:middle-beta} and Lemma~\ref{lem:middle-beta2}.
\end{proof}

\begin{problem}
A key assumption in our result is that the environmental function $r$ is piecewise constant across the jump discontinuity $x=X(t)$. It is interesting to inquire the level set asymptotics of 
\eqref{e.kpp.gen} 
when 
(i) $r= r(x-c_1t)$ is monotone in $x$; and where (ii) the limits $r(\pm\infty)$ exist and are positive, e.g., the problem when $r(t,x) = 1-a \psi(x-c_1t)$ was originally considered in \cite{holzer2014accelerated}, where $\psi(\xi)$ is a traveling wave profile which converges exponentially to $1$ as $\xi \to -\infty$ and to $0$ as $\xi \to +\infty$.
\end{problem}



%

\medskip
Received xxxx 20xx; revised xxxx 20xx; early access xxxx 20xx.
\medskip

\end{document}